\documentclass[12pt, a4paper, reqno]{amsart}
\usepackage[utf8]{inputenc}

\usepackage{amsfonts}
\usepackage{amsmath}
\usepackage{amssymb}
\usepackage{amsthm}
\usepackage[colorlinks = true,linkcolor=blue]{hyperref}
\usepackage{enumerate}
\usepackage{graphicx}
\usepackage[margin=1.4in]{geometry}
\usepackage{quiver}
\usepackage{amsrefs}

\makeatother
\newtheorem*{rep@theorem}{\rep@title}
\newcommand{\newreptheorem}[2]{%
\newenvironment{rep#1}[1]{%
 \def\rep@title{#2 \ref{##1}}%
 \begin{rep@theorem}}%
 {\end{rep@theorem}}}
\makeatother

\makeatletter
\newcommand*\bigcdot{\mathpalette\bigcdot@{.5}}
\newcommand*\bigcdot@[2]{\mathbin{\vcenter{\hbox{\scalebox{#2}{$\m@th#1\bullet$}}}}}
\makeatother

\newtheorem{counter}{subcounter}[section]
\newtheorem{cor}[counter]{Corollary}
\newtheorem{lemma}[counter]{Lemma}
\newtheorem{prop}[counter]{Proposition}
\newtheorem*{prop*}{Proposition}
\newreptheorem{prop}{Proposition}
\newtheorem{question}[counter]{Question}
\newtheorem{theorem}[counter]{Theorem}
\newtheorem{fact}[counter]{Fact}
\newreptheorem{theorem}{Theorem}

\theoremstyle{definition}\newtheorem{defn}[counter]{Definition}
\theoremstyle{definition}\newtheorem{example}[counter]{Example}
\theoremstyle{definition}\newtheorem{remark}[counter]{Remark}

\DeclareMathOperator{\Sa}{Sa}
\DeclareMathOperator{\conv}{conv}
\DeclareMathOperator{\Stab}{Stab}
\DeclareMathOperator{\proj}{Proj}
\DeclareMathOperator{\ext}{ext}
\DeclareMathOperator{\cS}{S}
\newcommand{\eps}{\varepsilon}
\newcommand{\C}{\mathbb{C}}
\newcommand{\R}{\mathbb{R}}

\newcommand{\Z}{\mathbb{Z}}
\newcommand{\T}{\mathbb{T}}
\newcommand{\cN}{\mathcal{N}}
\newcommand{\set}[1]{\left\{#1\right\}}
\newcommand{\bK}{\mathbf{K}}
\newcommand{\cK}{\mathcal{K}}
\newcommand{\aut}{\operatorname{Aut}}

\edef\restoreparindent{\parindent=\the\parindent\relax}
\usepackage{parskip}
\restoreparindent

\begin{document}

\title{Projectivity in topological dynamics}

\author{Jashan Bal}
\address{University of Waterloo, 200 University Avenue West, Waterloo, Ontario, N2L 3G1, Canada}
\email{j2bal@uwaterloo.ca}
\thanks{The author was partially supported by the Ontario Graduate Scholarship.}

\begin{abstract}
We study projectivity in the category of $G$-flows and affine $G$-flows for Polish groups $G$. We also introduce the notion of \emph{proximally irreducible} extensions between affine $G$-flows. Using this we provide a characterization of extreme amenability, strong amenability, and amenability for closed subgroups $H \leq G$ in terms of certain ``dynamical irreducibility'' properties of the Samuel compactification of $G/H$. We then apply this to answer an open question of Zucker by proving a structure theorem for when the universal minimal proximal flow of $G$ is metrizable or contains a comeager orbit. 

\end{abstract}

\maketitle

\setcounter{tocdepth}{1}
\tableofcontents
\section{Introduction}
A compact space $X$ is said to be \emph{projective} if for all compact spaces $Y,Z$, continuous maps $\phi\colon X \to Y$, and continuous surjections $\pi\colon Z \to Y$, there exists a continuous map $\psi\colon X \to Z$ satisfying $\pi \circ \psi = \phi$. The seminal paper of Gleason \cite{Gleason1958} characterized projective compact spaces as exactly the \emph{extremally disconnected} spaces, i.e., ones where the closure of every open subset remains open. Moreover, Gleason also proved the existence and uniqueness of \emph{projective covers} associated to every compact space.

In this paper we aim to further expand upon the theory of projectivity in topological dynamics. We first study projectivity in the category of $G$-flows. It follows from the proof due to Gleason \cite{Gleason1958} that a compact space $X$ is projective if and only if every \emph{irreducible extension} of $X$ is an isomorphism. Here an extension $\phi\colon Y \to X$ is said to be irreducible if for all proper closed subsets $C \subseteq Y$ we have $\phi(C) \neq X$. We can now consider irreducible extensions between $G$-flows where the extensions are also assumed to be $G$-equivariant. Then a $G$-flow $X$ is said to be \emph{$G$-extremally disconnected} ($G$-ED) if any irreducible extension of $X$ must be an isomorphism. These have been well studied in the literature (see \cites{AuslanderGlasner1977, BassoZucker2025,LeBoudecTsankov2025, Zucker2016, Zucker2021}). 

We can naturally weaken the conditions on an irreducible extension to ask that only proper subflows can't surject onto the entire codomain. In particular, an extension between $G$-flows $\phi\colon X \to Y$ is said to be \emph{$G$-irreducible} if for all proper subflows $C \subseteq X$ we have $\phi(C) \neq Y$. The maximal $G$-flows with respect to $G$-irreducible extensions turn out to be exactly the projective objects in the category of $G$-flows and are called $G$-projective flows. This was proved by Ball and Hagler in \cite{BallHagler1996} where they also showed that every $G$-flow has a unique maximal $G$-irreducible extension. Unsurprisingly, it turns out that the notion of $G$-ED does not coincide with $G$-projectivity in general (see the discussion right after Corollary \ref{Cor:ProjImpliesGED}). For example, when restricted to the class of minimal $G$-flows the only projective object is the universal minimal $G$-flow $M(G)$. However, unlike the theory of $G$-ED flows where we have a plethora of examples and known structure results, there are few known examples of $G$-projective flows. Our first main result is the following theorem which allows us to find examples of $G$-projective flows and gives a characterization of extreme amenability for closed subgroups in terms of the Samuel compactification of the coset space.
\begin{reptheorem}{Thm:ExAmenSubgp}
If $G$ is a Polish group and $H \leq G$ is a closed subgroup, then $H$ is extremely amenable if and only if $\Sa(G/H)$ is $G$-projective.
\end{reptheorem}
In the case when $\Sa(G/H)$ is minimal we see that $\Sa(G/H)$ is $G$-projective if and only if $\Sa(G/H)\simeq M(G)$. Then extreme amenability of $H$ follows by a result of Zucker \cite[Theorem 7.5]{Zucker2021}.

Moreover, we can show that in certain cases one can explicitly compute the \emph{$G$-projective cover} (see Theorem \ref{Thm:ProjCoverPointTrans}). We apply this to examples coming from structural Ramsey theory (see \cite{KechrisPestovTodorcevic2005}). In particular, in Proposition \ref{Prop:RamseyEx} we provide a complete characterization of metrizable topologically transitive $G$-projective flows when $G$ is equal to the automorphism group of a Fra\"{\i}ss\'e structure.

Using the compact right topological semigroup structure on $\Sa(G)$ we provide a correspondence between point transitive $G$-projective flows and idempotents in $\Sa(G)$. Moreover, in general we show $G$-projective flows are exactly the retracts of certain kinds of $G$-flows (see Corollary \ref{Cor:IdemAndProjPointTrans} and Proposition \ref{Prop:GxZProj}).

By Gelfand duality and properties of the Samuel compactification we see that $\Sa(G/H)$ is $G$-projective if and only if the right uniformly continuous bounded functions $\operatorname{RUC}^b(G/H)$ are injective in the category of unital commutative $G$-$C^*$-algebras with respect to $G$-equivariant $*$-homomorphisms. 

Injectivity for dynamical systems is an important and well studied property in operator algebras (see \cites{HadwinPaulsen2011,Hamana1985,KalantarKennedy2017,KennedyRaumSalomon2022}). Unlike the commutative setting, injectivity for general noncommutative $C^*$-algebras is defined with respect to unital completely positive maps instead of $*$-homomorphisms (see \cite{Hamana1979a} and \cite{Hamana1979b}). Motivated by this we look at the category of \emph{$G$-function systems} with morphisms given by $G$-equivariant unital positive maps. Kadison's representation theorem tells us that this is equivalent to studying projectivity in the category of \emph{affine $G$-flows} with morphisms as affine $G$-maps. Given a $G$-flow $X$, the set of regular Borel probability measures $P(X)$ on $X$ is an affine $G$-flow. With this we can show that the probability measures on the \emph{Furstenberg boundary}, $P(\Pi_s(G))$, is an affine $G$-projective flow. Our second main result here is an analogous characterization of amenability for closed groups.
\begin{reptheorem}{Thm:AmenSubgp}
    If $G$ is a Polish group and $H \leq G$ is a closed subgroup, then $H$ is amenable if and only if $P(\Sa(G/H))$ is affine $G$-projective. 
\end{reptheorem}
Dualizing we obtain that $H$ is amenable if and only if $\operatorname{RUC}^b(G/H)$ is injective in the category of $G$-function systems with respect to $G$-equivariant positive maps. This can also be viewed as a generalization of the know fact that $\C$ is injective in the category of $G$-function systems if and only if $G$ is amenable. 

 In analogy with $G$-projectivity we can show that projectivity in the category of affine $G$-flows corresponds to being maximal with respect to \emph{affine $G$-irreducible extensions}. Here an affine extension between two affine $G$-flows $\phi\colon K \to L$ is affine $G$-irreducible if for all proper affine subflows $C \subseteq K$ we have that $\phi(C)\neq L $. With basically the same argument that Ball and Hagler \cite{BallHagler1996} use to prove the existence of maximal $G$-irreducible extensions, we can prove that maximal affine $G$-irreducible extension exist (see Theorem \ref{Thm:AffProjCoverExist}). These are exactly the  \emph{affine $G$-projective covers}. For example, the affine $G$-projective cover of the trivial flow is $P(\Pi_s(G))$.

Motivated by the characterizations of amenability and extreme amenability for closed subgroups in terms of certain ``dynamical irreducibility'' properties of $\Sa(G/H)$, we aim to prove a similar result for strong amenability. Towards this we introduce a new notion of extensions for affine $G$-flows which we call \emph{proximally irreducible} (see Definition \ref{Def:ProxIrr}). For example, we can show that the probability measures on the universal minimal proximal flow $P(\Pi(G))$ has no nontrivial proximally irreducible extensions. Moreover, we obtain our desired characterization of strong amenability for closed subgroups.
\begin{reptheorem}{Thm:StrAmenSubgp}
    If $G$ is a Polish group and $H \leq G$ is a closed subgroup, then $H$ is strongly amenable if and only if $P(\Sa(G/H))$ has no nontrivial proximally irreducible extensions.
\end{reptheorem}
  Applying the above Theorem we can answer an open question of Zucker \cite[Question 7.6]{Zucker2021}. In doing so we obtain a characterization of when the universal minimal proximal flow has a comeager orbit or is metrizable. 
\begin{reptheorem}{Thm:StructureOfUMPF}
    Let $G$ be a Polish group.
    \begin{enumerate}
        \item $\Pi(G)$ has a comeager orbit if and only if there exists a maximal presyndetic strongly amenable closed subgroup $H \leq G$. Moreover in this case, $\Pi(G) \simeq \Sa(G/H)$.
        \item $\Pi(G)$ is metrizable if and only if there exists a maximal presyndetic co-precompact strongly amenable closed subgroup $H \leq G$. Moreover in this case, $\Pi(G) \simeq \widehat{G/H}$.
    \end{enumerate}
\end{reptheorem}
Unlike the case of $G$-irreducible and affine $G$-irreducible extensions, we were not able to show that maximal proximally irreducible extensions exists in general. The issue being that it is not clear if the composition of two proximally irreducible extensions remains proximally irreducible (see Question \ref{Ques:CompProxIrrExt}). However, when this obstacle disappears we can prove the existence and uniqueness of a maximal proximally irreducible extension. For example, if $X$ is a $G$-flow which is minimal or has a comeager orbit, then the affine $G$-flow $P(X)$ has a maximal proximally irreducible extension (see Theorem \ref{Thm:MaxProxIrrExt}). Using this we can explicitly compute the maximal proximal irreducible extension in certain cases (see Theorem \ref{Thm:MaxProxIrrOfPointTrans}). For example, we see that the maximal proximally irreducible extension of the trivial flow is $P(\Pi(G))$.

\subsection{Acknowledgements} I would like to thank Andy Zucker for all the helpful conversations, guidance, and feedback during this project.  
\section{Background}
 \subsection{Notation}
Throughout all topological space are assumed to be Hausdorff. Moreover, $G$ will always denote a Polish group, $1_G\in G$ the identity element, and $\mathcal{N}_G$ the collection of open neighbourhoods of $1_G$ in $G$. Some of the results proved also hold for arbitrary topological groups so we will often specify the assumption that $G$ is Polish when it necessary.
\subsection{Topological dynamics} See Auslander \cite{Auslander1988} or Glasner \cite{Glasner1976} for a detailed introduction to topological dynamics

 A \emph{$G$-flow} is a compact Hausdorff space $X$ with a specified continuous (left) action $\tau\colon G\times X \to X$. We write $g\cdot x$ or $gx$ to mean $\tau(g,x)$ when the action is understood. When the acting group is understood, we will sometimes write \emph{flow} to mean $G$-flow. Given $G$-flows $X$ and $Y$ we say a map $\phi\colon X \to Y$ is a \emph{$G$-map} if $\phi$ is continuous and $G$-equivariant, i.e., for every $x\in X$ and $g\in G$ we have $\phi(gx) = g\phi(x)$. Furthermore, if $\phi$ is surjective, then we say $\phi$ is an \emph{extension} or that $X$ is an extension of $Y$ when the map $\phi$ is understood. We say two $G$-flows are \emph{isomorphic} if there exists a bijective $G$-map between them

Let $X$ be a $G$-flow.  A \emph{subflow} of $X$ is a nonempty closed $G$-invariant subset $Y \subseteq X$. We say $X$ is a \emph{minimal} $G$-flow if it contain no proper subflows or equivalently every element $x\in X$ has dense orbit. By Zorn's lemma every $G$-flow contains a minimal subflow. It is classic result of Ellis that there exists a \emph{universal minimal flow}, denoted $M(G)$, such that $M(G)$ is a minimal $G$-flow and is an extension of all other minimal $G$-flows. Furthermore, $M(G)$ is unique up to isomorphism. The group $G$ is said to be \emph{extremely amenable} if every $G$-flow has a fixed point or equivalently $M(G)$ is trivial. 

We say a $G$-flow $X$ is \emph{point transitive} if there exists $x\in X$ with dense orbit. There exists a \emph{universal point transitive flow}, denoted $\Sa(G)$, with the following properties. We can identify $G \subseteq \Sa(G)$ as a dense subset and thus $1_G \in \Sa(G)$ has dense orbit. Also, for any $G$-flow $X$ with $x\in X$, there exists a $G$-map $\phi:\Sa(G) \to X$ such that $\phi(g) = gx$ for all $g\in G$. Moreover, $\Sa(G)$ is unique up to isomorphism. It also called called the \emph{Samuel compactification} of $G$ (see \cite{Samuel1948}). We will present a construction and some important properties of the Samuel compactification in the next section (see Section \ref{Sec:SamuelCompact}).

We say a $G$-flow $X$ is \emph{proximal} if for all $x,y\in X$ there exists a net $(g_\lambda)_\lambda \subseteq G$ such that $\lim_\lambda g_\lambda x = \lim_\lambda g_\lambda y$. There exists a \emph{universal minimal proximal flow}, denoted $\Pi(G)$, such that $\Pi(G)$ is a minimal proximal $G$-flow and it is an extension of any other minimal proximal $G$-flow. In addition, $\Pi(G)$ is unique up to isomorphism. See \cite{Glasner1976} for a proof of the existence of the universal minimal proximal flow. The group $G$ is said to be \emph{strongly amenable} if each proximal $G$-flow has a fixed point or equivalently $\Pi(G)$ is trivial. We say an extension between $G$-flows $\phi\colon X \to Y$ is \emph{proximal} if each fiber of $\phi$ consists of proximal pairs, i.e., for all $x,y\in X$ with $\phi(x) = \phi(y)$ there exists $(g_\lambda)_\lambda \subseteq G$ with $\lim_\lambda g_\lambda x = \lim_\lambda g_\lambda y$. Thus a $G$-flow $X$ is proximal if and only if the trivial extension $X \to \set{*}$ is proximal.  

If $X$ is a compact Hausdorff space, let $P(X)$ denote the collection of all regular Borel probability measures on $X$. By the Riesz representation theorem we may view $P(X) \subseteq C(X)^*$. With the induced weak$\ast$-topology we obtain that $P(X)$ is a compact Hausdorff space. We will identify $X$ as a closed subset of $P(X)$ where $x \in X$ is sent to the Dirac mass $\delta_x \in P(X)$. Notice that the extreme points of $P(X)$ are the exactly the set of Dirac masses. Furthermore, if $X$ is assumed to be a $G$-flow, then the action of $G$ on $X$ induces a continuous action of $G$ on $P(X)$ via continuous affine homeomorphisms. Therefore, $P(X)$ is also seen to be a $G$-flow. More generally we say a compact convex set $K$ is an \emph{affine $G$-flow} if there is some specified continuous action of $G$ on $K$ via affine homeomorphisms. An \emph{affine subflow} of $K$ is a $G$-invariant compact convex subset of $K$. Let $\ext(K)$ denote the extreme points of $K$. Notice that $\ext(K)$ is $G$-invariant and thus the closure $\overline{\ext}(K)$ is a $G$-flow. A surjective affine $G$-map between two affine $G$-flows is said to be an \emph{affine extension}. 

If $K$ is a compact convex subset of a locally convex topological vector space $E$, then there exists a unique surjective affine continuous map, called the \emph{barycenter map}, $\beta\colon  P(K) \to K$ such that
\[
f(\beta(\mu)) = \int_K fd\mu
\]
for all $f\in E^*$ and $\mu \in P(X)$ (see \cite{Phelps2001}). If $\beta(\mu) = x$ we say that $\mu$ is a \emph{representing measure} for $x$. In the case when $K$ is an affine $G$-flow then by uniqueness of the barycenter map we obtain that $\beta$ must be $G$-equivariant, i.e., an extension of affine $G$-flows. 

Given a compact convex set $K$ and a subset $C \subseteq K$, let $\conv(C)$ denote the set of convex combinations from elements of $C$ and $\overline{\conv}(C)$ the closure. Notice if $K$ is an affine $G$-flow and $C \subseteq K$ is $G$-invariant, then $\overline{\conv}(C)$ is an affine subflow. Recall the Krien-Milman theorem which states that for any compact convex set $K$ we have that $K = \overline{\conv}(\ext(K))$. Also, recall Milman's theorem which states that if $K$ is a compact convex set and $C \subseteq K$ is a closed subset such that $K = \overline{\conv}(C)$, then $\ext(K) \subseteq C$. For a proof of the above mentioned theorems see \cite{Phelps2001}.

We say a $G$-flow $X$ is \emph{strongly proximal} if $P(X)$ is a proximal $G$-flow. There exists a \emph{universal minimal strongly proximal flow}, denoted $\Pi_s(G)$, such that $\Pi_s(G)$ is a minimal strongly proximal $G$-flow and it is an extension of every other minimal strongly proximal $G$-flow. Moreover, $\Pi_s(G)$ is also unique up to isomorphism. See \cite{Glasner1976} for a proof of the existence of the universal minimal strongly proximal flow. Often the universal minimal strongly proximal flow is also called the \emph{Furstenberg boundary} of $G$. The group $G$ is said to be \emph{amenable} if each strongly proximal $G$-flow has a fixed point or equivalently $\Pi_s(G)$ is trivial. We say an extension between $G$-flows $\phi\colon X \to Y$ is \emph{strongly proximal} if for any $y \in Y$ and $\mu \in P(\phi^{-1}(\set{y})) \subseteq P(X)$ there exists a net $(g_\lambda)_\lambda \subseteq G$ such that 
\[
\lim_\lambda g_\lambda \mu  = \delta_x
\]
for some $x\in X$. It can be shown that $X$ is a strongly proximal $G$-flow if and only if the trivial extension $X \to \set{*}$ is strongly proximal. 
\subsection{Samuel compactification}\label{Sec:SamuelCompact} 
We present a construction of the Samuel compactification and some properties which will be needed later on. See \cite{KocakStrauss19797}, \cite{Zucker2021} and \cite{BassoZucker2025} for more details and proofs.

The following construction can be done in general for uniform spaces but we restrict our attention to spaces of the form $G/H$ where $G$ is a topological group and $H \leq G$ is a closed subgroup.
\begin{defn}
    A collection of subsets $\mathcal{F} \subseteq \mathcal{P}(G/H)$ has the \emph{near finite intersection property} if for all $A_1,\dots, A_n \in \mathcal{F}$ and open $ U \in \mathcal{N}_G$ we must have
    \[
    UA_1 \cap \cdots \cap UA_n \neq \emptyset.
    \]
    where $UA_j := \set{gkH: g\in U \text{ and } kH \in A_j}$.
\end{defn}
A \emph{near ultrafilter} on $G/H$ is a maximal collection of subsets of $G/H$ with the near finite intersection property. Let $\Sa(G/H)$ be the set of all near ultrafilters. Given $p\in \Sa(G/H)$ and $A \subseteq G/H$ we have that $A \in p$ if and only if $UA \in p$ for all $U \in \mathcal{N}_G$.

We endow $\Sa(G/H)$ with the topology whose basic open neighbourhoods are given by
\[
N_A:= \set{p\in \Sa(G/H): A\notin p}
\]
for $A \subseteq G/H$. Then $\Sa(G/H)$ is a compact Hausdorff space and we can embed $G/H \hookrightarrow \Sa(G/H)$ as a dense subset by associating each $gH \in G/H$ to the near ultrafilter $\widehat{gH} \in \Sa(G/H)$ given by
\[
\widehat{gH} := \set{A \subseteq G/H: gH \in \overline{A}}.
\]
Thus we will identify $G/H \subseteq \Sa(G/H)$. Moreover, $\Sa(G/H)$ will be a $G$-flow. Furthermore, if $X$ is a $G$-flow and $x \in X$ is an $H$-fixed point, then there exists a $G$-map $\phi\colon \Sa(G/H) \to X$ such that $\phi(gH) = gx$ for all $gH\in G/H$. This follows from the following result.
\begin{fact} \label{Fact:ExtUnifCtsMap}
    If $X$ is a compact space and $\phi\colon G/H \to X$ is a uniformly continuous function, then it extends uniquely to a continuous map $\Tilde{\phi}\colon \Sa(G/H) \to X$ such that $\Tilde{\phi}(p) = x$ if and only if for all open $W \ni x$ we have that
    \[
    \phi^{-1}(W) \in p.
    \]
\end{fact}
The following fact will also be implicitly used throughout the paper.
\begin{fact}
    If $G$ is a Polish group and $H\leq G$ is a closed subgroup, then $G/H \subseteq \Sa(G/H)$ is a comeager subset.
\end{fact}
Now if $H \leq G$ is a closed subgroup, then given any $p \in \Sa(H)$ we have that the collection $\set{UA: A \in p, U \in \mathcal{N}_G}$ extends to a unique near ultrafilter on $G$. This induces an embedding $\iota: \Sa(H) \to \Sa(G)$ where we have that $p \in \iota(\Sa(H))$ if and only if $UH \in p$ for all $U \in \mathcal{N}_G$. Thus we identify $\Sa(H) \subseteq \Sa(G)$ as a closed subset. 

 There is another, more classical, method of constructing the Samuel compactification using Gelfand duality. Let $\operatorname{RUC}^b(G/H)$ be the collection of right uniformly continuous bounded functions on $G/H$, i.e., a bounded function $f\colon G/H \to \C$ is \emph{right uniformly continuous} if for all $\eps > 0$ there exists an open neighbourhood $V$ of $H \in G/H$ such that
\[
xy^{-1} \in V \implies |f(x) - f(y)| < \eps.
\]
When endowed with the supremum norm $\operatorname{RUB}^b(G/H)$ is a unital commutative $C^\ast$-algebra and define $\Sa(G/H)$ be the spectrum of said $C^\ast$-algebra. Notice $G/H$ embeds into $\Sa(G/H)$ via the evaluational functionals. In particular, $\widehat{gH} \in \Sa(G/H)$ is defined via
\[
\widehat{gH}(f) := f(gH)
\]
for all $f\in \operatorname{RUC}^b(G/H)$. 

From the above discussion it follows that $\Sa(G)$ is the universal point transitive $G$-flow. Indeed, $1_G\in \Sa(G)$ is seen to have dense orbit and for any $G$-flow $X$ with $x\in X$ there exists a $G$-map $\phi \colon \Sa(G) \to X$ with $\phi(g)= gx$ for all $g\in G$. Moreover, we have that $\Sa(G)$ is a \emph{compact right topological semigroup}, i.e., $\Sa(G)$ is a semigroup such that for all $p \in \Sa(G)$ the map $\Sa(G) \ni q \mapsto q\cdot p$ is continuous. Lastly if $X$ is a $G$-flow, then there exists a semigroup action of $\Sa(G)$ on $X$ which extends the action of $G$. This action is seen to be continuous if we fix the $X$-coordinate, i.e., if $x\in X$ is fixed, then the map $\Sa(G)\ni p\mapsto px $ is continuous. 
\subsection{Kadison's representation theorem}

Let $A$ be a unital commutative $C^*$-algebra. A norm closed unital self-adjoint subspace $E \subseteq A$ is said to be a \emph{function system}. A linear map $\phi\colon E_1 \to E_2$ between two function systems is said to be \emph{positive} if for all $x\in E_1$
\[
x\geq 0 \implies \phi(x) \geq 0.
\]
A unital linear map $\phi\colon E_1 \to E_2$ is said to be a \emph{unital order isomorphism} if $\phi$ is bijective and for all $x\in E_1$ we have
\[
x\geq 0 \iff \phi(x) \geq 0.
\]
If $E$ is a function system, then any unital positive map $\phi\colon E \to \C$ is said to be a \emph{state} on $E$. The collection of all states on $E$, denoted $S(E)$ and called the \emph{state space}, is a compact convex set when equipped with the weak-* topology. 

Now Kadison's representation theorem, originally proved in \cite{Kadison1951}, states that the category of function systems with unital positive maps as morphisms is dually equivalent to the category of compact convex sets with continuous maps as morphisms. Also see the book of Alfsen \cite{Alfsen1971} for another reference. More precisely we have the following.

Let $K$ be a compact convex set. Then $A(K)$, the set of all continuous affine functions on $K$, is a function system when viewed as a subspace of $C(K)$ and the natural evaluation map $\phi\colon K \to S(A(K))$ is an affine homeomorphism. Conversely, if $E$ is a compact convex set, then the evaluation map $\psi\colon E \to A(S(E))$ is a unital order isomorphism. 

Let $E_1,E_2$ be function systems. If $\psi\colon E_1 \to E_2$ is a unital positive map, then $\phi \colon S(E_2) \to S(E_1)$ defined via
\[
\phi(f)(x) := f(\psi(x))
\]
for $f\in S(E_2)$ and $x\in E_1$ is a continuous affine map. Moreover, if $\psi$ is injective (surjective), then $\phi$ is surjective (injective). Similarly, if $K_1, K_2$ are compact convex sets and $\phi: K_1 \to K_2$ is a continuous affine map, then $\psi: A(K_2)\to A(K_1)$ defined via
\[
\psi(f)(x) := f(\phi(x))
\]
for $f \in A(K_2)$ and $x \in K_1$ is a unital positive map. Moreover, if $\phi$ is injective (surjective), then $\psi$ will be surjective (injective).

A \emph{$G$-function system} is a function system $E$ with some specified continuous action of $G$ on $E$ via unital order isomorphisms. We can then naturally define an action of $G$ on $S(E)$ via 
\[
g\cdot f (x) := f (g^{-1}x)
\]
for $f \in S(E)$, $x \in E$, and $g\in G$. It follows that $S(E)$ is an affine $G$-flow. Conversely, if $K$ is an affine $G$-flow, then $A(K)$ is a $G$-function system with the action given by 
\[
g\cdot f(x) := f(g^{-1}x)
\]
where $f \in A(K)$, $x \in K$, and $g\in G$. It follows that the duality described above induces a duality between the category of $G$-function systems with morphisms given by $G$-equivariant unital positive maps and the category of affine $G$-flows with morphisms given as affine $G$-maps. 

\section{$G$-projective flows}

\begin{defn}
     A $G$-flow $X$ is said to be \emph{$G$-projective} if for any $G$-flows $Y$, $Z$, $G$-maps $\phi\colon X \to Y$, and surjective $G$-maps $\pi\colon Z \to Y$, there exists a $G$-map $\psi\colon X \to Z$ such that $\pi \circ \psi = \phi$.
\end{defn}
The above can be summarized by the following commutative diagram:
\[\begin{tikzcd}
	&& Z \\
	\\
	X && Y
	\arrow["\pi", two heads, from=1-3, to=3-3]
	\arrow["{\exists \psi}", dotted, from=3-1, to=1-3]
	\arrow["\phi"', from=3-1, to=3-3]
\end{tikzcd}\]
We will not explicitly define injectivity here but the following is an immediate consequence of Gelfand duality.
\begin{cor}
    A $G$-flow $X$ is $G$-projective if and only if the $G$-$C^{\ast}$-algebra $C(X)$ is injective in the category of unital commutative $G$-$C^*$-algebras with respect to $G$-equivariant $*$-homomorphisms. 
\end{cor}
\begin{example}
    Our first example of a $G$-projective flow is the universal point transitive $G$-flow, i.e., the Samuel compactification $\Sa(G)$.

    Suppose $Y$ and $Z$ are $G$-flows, $\phi\colon \Sa(G) \to Y$ a $G$-map, and $\pi\colon Z \to Y$ a surjective $G$-map. Fix some $z \in \pi^{-1}(\set{\phi(1_G)})$ and notice by the universal property of $\Sa(G)$ there exists a $G$-map $\psi\colon \Sa(G) \to Z$ such that $\psi(g) = gz$ for all $g\in G$. Checking on the dense subset $G \subseteq \Sa(G)$ we obtain that $\pi \circ \psi = \phi$. 
\end{example}
Given a $G$-flow $X$ we have a natural semigroup action of $\Sa(G)$ on $X$. Given a point $x\in X$ we associate to it the subset $S_x \subseteq \Sa(G)$ given by
\[
S_x := \set{p \in \Sa(G): px = x}.
\]
Then $S_x$ is a closed sub-semigroup of $\Sa(G)$. The following is a crucial result which allows us describe $S_x$ under certain structural assumptions on the orbit of $x \in X$.
\begin{lemma}\label{Lem:Stab}
    Let $G$ be a Polish group and $X$ a $G$-flow. If an element $x_0 \in X$ has nonmeager orbit, then 
    \[
    S_x = \Sa(\Stab(x_0)).
    \]
\end{lemma}
\begin{proof}
    Let $H := \Stab(x_0) \leq G$ and $\phi\colon \Sa(G) \to X$ be the extension given by $\phi(g) = gx_0$. It follows that
    \[
    S_x = \phi^{-1}(\set{x_0}).
    \]
    By continuity of $\phi$ we obtain that $\Sa(H) \subseteq S_x$. As $\phi$ is the continuous extension of the uniformly continuous map $\psi: G \to X$ given by $\psi(g) = gx_0$ we see that $\phi(p) = x_0$ if and only if for all open $W \ni x_0$ we have that $\psi^{-1}(W) \in p$.
    
    Now fix $p \in S_x$ and let $U \in \cN_G$. Then Effros's theorem \cite{Effros1965} implies $Ux_0$ is relative open in $Gx_0$. Thus there exists an open subset $W \ni x_0$ such that $W \cap Gx_0 = Ux_0$. It follows that
    \[
    UH = \psi^{-1}(Ux_0) = \psi^{-1}(W) \in p.
    \]
    As this holds for all $U \in N_G$ we obtain that $p \in \Sa(H)$.
\end{proof}
    For Lemma \ref{Lem:Stab} the case when $X = \Sa(G/H)$ for a closed subgroup $H \leq G$ was shown in \cite[Proposition 6.4]{Zucker2021}. In fact we will primarily only be using this case for the results presented here. 
\begin{example}\label{Ex:IrrRot}
    We present an example showing the necessity of the nonmeager assumption in Lemma \ref{Lem:Stab}.
    
    Let $\Z \curvearrowright \T$ be an irrational rotational. Then $\T$ is a free minimal $\Z$-flow. Fix $x \in \T$ and let $(m_n)_n \subseteq \Z$ be a sequence such that $|m_n| \geq n$ and 
    \[
    |m_n \cdot x - x| < \frac{1}{n}.
    \]
    If $\rho \in \beta\Z$ is a cluster point of $(m_n)_n$, then $\rho$ is necessarily a non-principal ultrafilter and $\rho \cdot x = x$ by continuity. Thus 
    \[
    \set{0} = \beta(\operatorname{Stab}(x)) \subsetneq S_x.
    \]
    is a strict subset.
    
    This particular example can be easily generalized to show any free minimal $G$-flow $X$ for a non-compact Polish group $G$ must have all meager orbits. Indeed, pick some $U \in \cN_G$ which cannot cover $G$ by finitely many translates. Then for any $x \in X$ we can find a net $(g_\lambda)_\lambda \subseteq G\setminus U$ such that $\lim_\lambda g_\lambda x = x$. By passing to a subnet we obtain that $S_x \neq \set{1_G}$.
\end{example}
\begin{theorem}\label{Thm:ExAmenSubgp}
    If $G$ is a Polish group and $H \leq G$ a closed subgroup, then $H$ is extremely amenable if and only if $\Sa(G/H)$ is $G$-projective. 
\end{theorem}
\begin{proof}
    $(\Rightarrow)$ Let $X$ and $Y$ be $G$-flows, $\phi\colon \Sa(G/H) \to X$ a $G$-map, and $\pi\colon Y \to X$ a surjective $G$-map. Notice $x:= \phi(H) \in X$ is an $H$-fixed point. It follows that $\pi^{-1}(\set{x}) \subseteq Y$ is a nonempty $H$-flow. Since $H$ is extremely amenable, there exists an $H$-fixed point $y\in \pi^{-1}(\set{x})$. By the universal property of $\Sa(G/H)$, there exists a $G$-map $\psi\colon  \Sa(G/H) \to Y$ such that $\phi(gH) = gy$ for every $g\in G$. It follows that for every $gH\in G/H$ we have
    \[
    \pi(\psi(gH)) = \pi(gy) = g\pi(y) = gx = \phi(gH).
    \]
    Thus $\pi\circ \psi = \phi$ as $G/H \subseteq \Sa(G/H)$ is a dense subset. 
    
    $(\Leftarrow)$ It suffices to show $\Sa(H)$ has an $H$-fixed point. Consider the $G$-map $\phi\colon  \Sa(G) \to \Sa(G/H)$ given by $\phi(g) = gH$. As $H \in \Sa(G/H)$ has comeager orbit, Lemma \ref{Lem:Stab} gives us that $\phi^{-1}(\set{H}) = \Sa(H)$. Since $\Sa(G/H)$ is $G$-projective there exists a $G$-map $\psi\colon  \Sa(G/H) \to \Sa(G)$ such that $\phi\circ\psi = \text{id}$. Hence, $\psi(H) \in \phi^{-1}(\set{H})= \Sa(H)$ is an $H$-fixed point. 
\end{proof}
The following can be seen as a generalization of Veech's theorem which states that any locally compact group $G$ acts freely on $\Sa(G)$.
\begin{prop}
    If $G$ is a locally compact group and $X$ is a $G$-projective flow, then $G$ acts freely on $X$.
\end{prop}
\begin{proof}
    By Veech's theorem, see \cite[Theorem 2.2.1]{Veech1977}, $G$ acts freely on $\Sa(G)$. Now let $\set{\ast}$ be the trivial $G$-flow and consider the $G$-maps $X \to \set{\ast}$ and $\Sa(G) \to \set{\ast}$. Since $X$ is $G$-projective, there exists a $G$-map $\phi\colon  X \to \Sa(G)$. It follows that for every $x\in X$ we have
    \[
    \Stab(x) \subseteq \Stab(\phi(x))= \set{1_G}. \qedhere
    \]
\end{proof}
\begin{remark}\label{Rmk:UniversalProjFlows}
    The above proof shows every $G$-projective flow $X$ is \emph{universal}, i.e., given any minimal $G$-flow $Y$, there exists a surjective $G$-map $\psi\colon X \to Y$. However, the converse is not true. Indeed, in \cite[Theorem 3.5.8]{GutmanNguyenVanThe2015} it was shown that for $G = S_\infty$, there exists a closed subgroup $H \leq G$ isomorphic to $\Z$ such that every $G$-flow has an $H$-fixed point, i.e., $\Sa(G/H)$ is universal. However, by Theorem \ref{Thm:ExAmenSubgp} alongside the fact that $\Z$ is not extremely amenable, we have that $\Sa(G/H)$ is not $G$-projective. 
\end{remark}
\subsection{$G$-projective covers}

\begin{defn}
    Let $X$ and $Y$ be $G$-flows. An extension $\phi\colon X \to Y$ is said to be \emph{$G$-irreducible} if for all proper subflows $Z \subseteq X$ we must have that $\phi(Z) \neq Y$. 
\end{defn}
\begin{defn}
    Let $X$ be a $G$-flow. A pair $(Y, \phi)$ is a \emph{$G$-projective cover} of $X$ if $\phi\colon Y \to X$ is a $G$-irreducible extension and given any other $G$-flow $Z$ with a $G$-irreducible extension $\pi\colon  Z \to X$, there exists a $G$-map $\psi\colon  Y \to Z$ such that $\pi \circ \psi = \phi$.
\end{defn}
The following is a result of Ball and Hagler from \cite{BallHagler1996} showing the existence and uniqueness of $G$-projective covers. One could also apply the main result in \cite{GutmanLi2013} to prove uniqueness. 
\begin{theorem}[\cite{BallHagler1996}]\label{Thm:ExistenceProjCover}
    If $X$ is a $G$-flow, then there exists a unique $G$-projective cover $(Y,\phi)$ of $X$. Moreover, $Y$ will be a $G$-projective flow. 
\end{theorem}
Given a $G$-flow $X$ we write $\operatorname{Proj}_G(X)$ to denote the (unique) $G$-projective cover of $X$ and we often omit writing the given $G$-irreducible extension from $\operatorname{Proj}_G(X)$ to $X$. The main idea behind the proof of Theorem \ref{Thm:ExistenceProjCover} as seen in \cite{BallHagler1996} is to show that every $G$-flow admits a $G$-irreducible extension which itself admits no nontrivial $G$-irreducible extensions. This maximal $G$-irreducible extension must be unique and equal to the $G$-projective cover. The following are seen to be consequences of their proof.
\begin{cor}[\cite{BallHagler1996}]\label{Cor:GIrrExtGProjCover}
    If $X$ is a $G$-flow, then $X$ is $G$-projective if and only if $X$ has no nontrivial $G$-irreducible extensions, i.e., if $Y$ is $G$-flow and $\phi\colon Y \to X$ is a $G$-irreducible extension, then $\phi$ must be a homeomorphism. Moreover, if $Z$ is $G$-projective flow and a $G$-irreducible extension of $X$, then $Z \simeq \proj_G(X)$.
\end{cor}
\begin{example}
As was shown in \cite{BallHagler1996}, the universal minimal $G$-flow $M(G)$ is also $G$-projective. Indeed, if $\phi\colon X \to M(G)$ is a $G$-irreducible extension, then $X$ must be a minimal $G$-flow. The universality of $M(G)$ implies $\phi$ must be an isomorphism. Thus Corollary \ref{Cor:GIrrExtGProjCover} tells us that $M(G)$ is $G$-projective.

 We also claim that for any minimal $G$-flow $X$, we have $\proj_G(X) \simeq M(G)$. Indeed, there exists an extension $\phi\colon M(G) \to X$. Notice $\phi$ has to be a $G$-irreducible extension as $M(G)$ contains no proper subflows. Therefore, Corollary \ref{Cor:GIrrExtGProjCover} implies that $\proj_G(X) \simeq M(G)$.
\end{example}
\subsection{Topometric structure and Gleason complete flows}
    \begin{defn}
          A $G$-flow $X$ is \emph{$G$-extremely disconnected} ($G$-ED) if for all open subsets $A \subseteq X$ and $U \in \mathcal{N}_G$ we have $\overline{A} \subseteq \operatorname{Int}(\overline{UA})$.
    \end{defn}
    As shown in \cite{Zucker2021}, given any $G$-flow $X$ we can take its \emph{Gleason completion}. This is a $G$-ED flow $\operatorname{S}_G(X)$ with an \emph{irreducible extension} $\phi\colon \cS_G(X) \to X$, i.e., for all proper closed subsets $C \subseteq \operatorname{S}_G(X)$ we have that $\phi(C) \neq X$. Moreover, $\cS_G(X)$ is universal in the sense that for any $G$-flow $Y$ and irreducible extension $\pi\colon Y \to X$, there exists a $G$-map $\psi\colon  \cS_G(X) \to Y$ such that $\pi \circ \psi = \phi$.
    
    Moreover, from \cite{Zucker2021}, we see that a $G$-flow $X$ is $G$-ED if and only if $\cS_G(X) \simeq X$, i.e., $X$ has no nontrivial irreducible extensions. 
    \begin{cor}\label{Cor:ProjImpliesGED}
           If $X$ is a $G$-projective flow, then $X$ is also $G$-ED.
    \end{cor}
    \begin{proof}
        Let $\phi\colon  \operatorname{S}_G(X) \to X$ be the Gleason completion of $X$. As $\phi$ is irreducible it must also be $G$-irreducible. Hence, Corollary \ref{Cor:GIrrExtGProjCover} implies that $\phi$ is an isomorphism and thus $X$ is $G$-ED.
    \end{proof}
 Note the converse to the above statement is not true in general. For example, let $G$ be a Polish group with $H\leq G$ a closed subgroup. Then $\Sa(G/H)$ is always $G$-ED by \cite[Section 3.2]{Zucker2021} while it is $G$-projective if and only if $H$ is extremely amenable by Theorem \ref{Thm:ExAmenSubgp}. 
 \subsection{Point transitive case} \label{Sec:PointTrans}
We sketch out the argument in \cite{BallHagler1996} showing how to compute the $G$-projective cover of a point transitive $G$-flow. We also make use of the Samuel compactification $\Sa(G)$.

Let $X$ be a point transitive $G$-flow and fix $x \in X$ with dense orbit. Consider the extension $\phi\colon \Sa(G) \to X$ given by $\phi(g) = gx$ for all $g\in G$. By Zorn's lemma there exists a subflow $Z \subseteq \Sa(G)$ minimal with respect to $\phi(Z) = X$. Therefore, $\phi|_Z\colon Z \to X$ is a $G$-irreducible extension. We claim that $\proj_G(X) \simeq Z$. Towards this let $Y$ be a $G$-flow and $\pi\colon Y \to X$ a $G$-irreducible extension. Fix $y \in \pi^{-1}(\set{x})$ and consider the extension $\psi\colon  \Sa(G) \to Y$ given via $\psi(g) = gy$ for all $g\in G$. It follows that $\psi|_Z\colon Z \to Y$ is a $G$-map satisfying $\pi\circ \psi|_Z = \phi|_Z$ thus showing $\proj_G(X) \simeq Z$. By uniqueness of the $G$-projective cover the the above construction of $Z$ does not depend (up to isomorphism) on the choice of the point $x\in X$ with dense orbit.

Under certain conditions we can explicitly compute the $G$-projective covers of point transitive $G$-flows. We will need the following structure result of $M(G)$ shown in \cite{Zucker2021}.

\begin{theorem}[\cite{Zucker2021}]
    If $G$ is a Polish group such that $M(G)$ has a comeager orbit, then $M(G) \simeq \Sa(G/G^*)$ for some extremely amenable closed subgroup $G^* \leq G$.
\end{theorem}
\begin{remark}
    The above result was first shown for the setting when $M(G)$ is metrizable in \cite{MellerayNguyenVanTheTsankov2016} where they additionally obtain that the extremely amenable closed subgroup $G^* \leq G$ is also co-precompact, i.e., $\Sa(G/G^*)$  is isomorphic to the completion $\widehat{G/G^*}$. Moreover, $M(G)$ metrizable implies it has a comeager orbit by \cite{BenYaacovMellerayTsankov2017}.
\end{remark}

\begin{theorem}\label{Thm:ProjCoverPointTrans}
    Let $G$ be a Polish group and $X$ a point transitive $G$-flow. If $x\in X$ has comeager orbit and the universal minimal flow of $H:= \Stab(x)$ has a comeager orbit, then $\proj_G(X) \simeq \Sa(G/H^*)$ where $H^* \leq H$ is an extremely amenable closed subgroup such that $M(H) \simeq \Sa(H/H^*)$.
\end{theorem}
\begin{proof}
We begin by first showing the extension $\phi\colon \Sa(G/H) \to X$ given by $\phi(gH) = gx$ is $G$-irreducible. We claim $\phi^{-1}(\set{x}) = \set{H}$. It suffices to show for any $\rho \in \phi^{-1}(\set{x})$ and $U \in \mathcal{N}_G$ we have $UH \in \rho$. By Effros's theorem there exists an open subset $W \subseteq X$ such that $Ux = W\cap Gx$. It follows that
\[
UH = \set{gH: gx \in Ux} = \set{gH: gx \in W} \in \rho.
\]
If $Y \subseteq \Sa(G/H)$ is a subflow such that $\phi(Y) = X$, then we must have that $Y\cap \phi^{-1}(\set{x})$ is nonempty. Therefore, $H \in Y$ and it follows that $Y = \Sa(G/H)$. 

By Theorem \ref{Thm:ExAmenSubgp} $\Sa(G/H^*)$ is $G$-projective. As the composition of two $G$-irreducible extensions remains $G$-irreducible it suffices, by Corollary \ref{Cor:GIrrExtGProjCover}, to show that $\Sa(G/H^*)$ is a $G$-irreducible extension of $\Sa(G/H)$.

Let $\eta\colon  G/H^* \to \Sa(G/H)$ be the uniformly continuous map given by $\eta(gH^*) = gH$ for all $g\in G$ and $\psi\colon  \Sa(G/H^*) \to \Sa(G/H)$ be its continuous extension. We claim that $\psi^{-1}(\set{H}) = \Sa(H/H^*)$. To show $\psi^{-1}(\set{H}) \subseteq \Sa(H/H^*)$ it suffices to show for any $\rho \in \psi^{-1}(\set{H})$ and $U \in \mathcal{N}_G$ we have $U\cdot (H/H^*) \in \rho$. Using Effros's theorem and the fact that $H \in \Sa(G/H)$ has comeager orbit there exists an open subset $W \subseteq \Sa(G/H)$ such that $W\cap G/H = UH$. It follows that
\[
U\cdot (H/H^*) = \eta^{-1}(UH) = \eta^{-1}(W) \in \rho.
\]
The reverse inclusion is clear. Now suppose $Y \subseteq \Sa(G/H^*)$ is a subflow such that $\phi(Y) = \Sa(G/H)$. It follows that $$Y \cap \psi^{-1}(\set{H}) = Y\cap \Sa(H/H^*)$$must be nonempty. As $\Sa(H/H^*) \simeq M(H)$ is a minimal $H$-flow it not hard to see that implies $\Sa(H/H^*) \subseteq Y$. As $H^* \in \Sa(G/H^*)$ has dense orbit we obtain that $Y = \Sa(G/H^*)$. Therefore, $\psi$ is a $G$-irreducible extension and thus $\proj_G(X) \simeq \Sa(G/H^*)$.
\end{proof}
We can still compute the $G$-projective of a point transitive $G$-flow $X$ where $x_0 \in X$ has comeager orbit without any assumptions on the metrizability of the universal minimal flow of $\Stab(x_0)$. In particular, viewing $M(\Stab(x_0)) \subseteq \Sa(G)$ as a closed subset we will have that $$\proj_G(X) \simeq \overline{G\cdot M(\Stab(x_0))}.$$ This will be proven later in the section on proximally irreducible flows after introducing additional techniques (see Proposition \ref{Prop:ProjCoverPointTransGen}).
\begin{prop}
    Let $G$ be a Polish group and $X$ a point transitive $G$-flow. If $x_0 \in X$ has comeager orbit, then $\proj_G(X)$ is isomorphic to $\overline{G\cdot M(\operatorname{Stab}(x_0))}$ where we identify $M(\operatorname{\Stab}(x_0)) \subseteq \Sa(G)$ as a closed subset. 
\end{prop}

We now show how Theorem \ref{Thm:ProjCoverPointTrans} can be used to compute $G$-projective covers in structural Ramsey theory. Some amount of background on Fra\"{\i}ss\'e theory is assumed (see \cite{Hodges1993} and \cite{KechrisPestovTodorcevic2005}). 

Let $L$ be a countable language, $\bK$ a Fra\"{\i}ss\'e $L$-structure on $\omega$, $\cK$ the Fra\"{\i}ss\'e class of $\bK$, $\operatorname{Fin}(\bK)$ the set of finite substructures of $\bK$, and $G:= \aut(\bK)$. Suppose $L^{\ast}$ is a countable language containing $L$ and $\cK^{\ast}$ is a Fra\"{\i}ss\'e class in $L^{\ast}$ which is also a reasonable precompact expansion of $\cK$. We define $X_{\cK^\ast}$ to be set of all expansions of $\bK$ to $L^\ast$-structures in a compatible way, more explicitly we have
\[
X_{\cK^\ast} := \set{\langle \bK, \vec{S}\rangle: \langle A,\vec{S}|_A \rangle \in \cK^* \text{ for every } A\in \operatorname{Fin}(\bK)}.
\]
where $<{\bK},{\vec{S}}>$ denotes an expansion of $\bK$ to an $L^\ast$-structure by interpreting $\vec{S} := L^\ast \setminus L$. We see that $X_{\cK^\ast}$ is a compact Hausdorff space with the topology induced by a basis of clopen sets of the form
\[
N_{A^\ast} := \set{<{\bK},{\vec{S}}> \in X_{\cK^\ast}: <{A},{\vec{S}|_A}> = A^\ast}
\]
where $A \in \operatorname{Fin}(\bK)$ and $A^\ast \in \cK^\ast$ is an expansion of $A$. Moreover, $X_{\cK^\ast}$ is a $G$-flow when endowed with the logic action. In particular, given $g\in G$, $\bK' \in X_{\cK^\ast}$, and a relational symbol $R \in L^\ast$ with arity $n$, we define $g\bK' \in X_{\cK^\ast}$ by letting
\[
(a_1,\dots, a_n) \in R^{g\bK'} \iff (g^{-1}(a_1),\dots, g^{-1}(a_n)) \in R^{\bK'}.
\]
Notice we can also view $X_{\cK^\ast} \simeq \Sa(G/G^*)$ where $G^* := \aut(\bK^*) \leq G$ is the automorphism group of the Fra\"{\i}ss\'e structure associated to $\cK^\ast$. 

Suppose $\cK^\ast$ has finite Ramsey degrees or equivalently $G^\ast$ has metrizable universal minimal flow. It follows that there exists a reasonable precompact expansion Fra\"{\i}ss\'e expansion $\cK^{\ast\ast}$ of $\cK^\ast$ in a countable language $L^{\ast\ast}$ containing $L^\ast$ such that the pair $(\cK^{\ast}, \cK^{\ast\ast})$ has the expansion property and $\cK^{\ast\ast}$ has the Ramsey property. Moreover, we must have $M(G^\ast) \simeq \Sa(G^\ast/G^{\ast\ast})$ where $G^{\ast\ast} := \aut(\bK^{\ast\ast})$ is the automorphism group of the Fra\"{\i}ss\'e structure associated to $\cK^{\ast\ast}$. See \cite[Theorem 5.7 and 8.14]{Zucker2016} for a reference to the above claimed facts. 

By Theorem \ref{Thm:ProjCoverPointTrans} we obtain that
\[
\proj_G(X_{\cK^\ast}) \simeq X_{\cK^{\ast\ast}}.
\]
For example let $G = S_\infty := \aut(\omega)$ and $X = \operatorname{Graphs}(\omega)$ be the set of all graphs on vertex set $\omega$. Then by \cite{KechrisPestovTodorcevic2005}, we have that $\proj_G(X)$ is isomorphism to $\operatorname{LOGraphs}(\omega)$, the set of all linearly ordered graphs on vertex set $\omega$.

Now if $X$ is a metrizable topological transitive $G$-projective flow, then we know $X$ must be $G$-ED and hence by \cite[Proposition 3.8]{Zucker2021}, $X \simeq X_{\cK^*}$ for some reasonable precompact Fra\"{\i}ss\'e expansion $\cK^\ast$ of $\cK$ in a countable language $L^\ast$ containing $L$. Then by Theorem \ref{Thm:ExAmenSubgp}, we must have that $G^\ast:= \aut(\bK^*)$ is extremely amenable and hence by the KPT correspondence (see \cite{KechrisPestovTodorcevic2005}) $\cK^{\ast}$ has the Ramsey property. In particular, we have shown the following.
\begin{prop}\label{Prop:RamseyEx}
    If $G= \aut(\bK)$ is the automorphism group of a Fra\"{\i}ss\'e structure $\bK$, then a metrizable topological transitive $G$-projective flow is exactly of the form $X_{\cK^\ast}$ for some reasonable precompact Fra\"{\i}ss\'e expansion of $\cK$ with the Ramsey property.
\end{prop}
We now present a characterization of point transitive $G$-projective flows through the use of the compact right topological semigroup structure on $\Sa(G)$.
\begin{defn}
    Let $Y$ be a $G$-flow. A $G$-flow $X$ is a \emph{retract of $Y$} if $X \subseteq Y$ and there exists a $G$-map $\phi\colon Y \to X$ such that $\phi|_X = \text{id}$.
\end{defn}
\begin{prop}\label{Prop:ProjPointTransCRTS}
   If $X$ is a $G$-flow, then $X$ is a point transitive $G$-projective flow if and only if $X$ is (up to isomorphism) a retract of $\Sa(G)$. 
\end{prop}
\begin{proof}
    $(\Rightarrow)$ Fix $x\in X$ with dense orbit and consider the surjective $G$-map $\phi\colon \Sa(G) \to X$ given by $\phi(g) = gx$. Let $S_x:= \phi^{-1}(\set{x})$ and notice $X$ is isomorphic to a subflow $Y \subseteq \Sa(G)$ minimal with respect to $Y\cap S_x \neq \emptyset$. As $Y \cap S_x$ is a nonempty closed sub-semigroup of $\Sa(G)$, by the Ellis-Numakura lemma, there exists an idempotent $\rho \in Y \cap S_x$. As $\Sa(G)\cdot \rho \subseteq Y$ is a subflow which surjections on $X$, by minimality of $Y$ we must have that $\Sa(G)\cdot \rho = Y$. Let $r_\rho\colon  \Sa(G) \to \Sa(G)$ denote the continuous $G$-map $r_\rho(q) = q\cdot \rho$. We obtain that $r_\rho\colon  \Sa(G) \to Y$ is the required retract. 
    
    $(\Leftarrow)$ Assume $X \subseteq \Sa(G)$ is subflow and $\phi\colon \Sa(G) \to X$ is a retract. Suppose $Y$ and $Z$ are $G$-flows, $\psi\colon  X \to Y$ a $G$-map, and $\pi\colon Z \to Y$ a surjective $G$-map. Notice $\psi \circ \phi\colon  \Sa(G) \to Y$ is a $G$-map and $G$-projectivity of $\Sa(G)$ implies there exists a $G$-map $\eta\colon  \Sa(G) \to Z$ such that $\pi\circ \eta = \psi \circ \phi$. It follows that $\eta|_X\colon X \to Z$ is a $G$-map satisfying 
    \[
    \pi \circ \eta|_X = \psi \circ \phi|_X = \psi.
    \]
    Hence, $X$ is $G$-projective.
\end{proof}
The converse direction of the above proof shows the following general statement.
\begin{cor}\label{Cor:RetractProj}
    Let $Y$ be a $G$-projective flow. If $X$ is $G$-flow and a retract of $Y$, then $X$ is also $G$-projective. 
\end{cor}
Moreover, the proof of Proposition \ref{Prop:ProjPointTransCRTS} shows the following.
\begin{cor} \label{Cor:IdemAndProjPointTrans}
    The point transitive $G$-projective flows are (up to isomorphism) exactly of the form $\Sa(G)\cdot \rho$ for an idempotent $\rho \in \Sa(G)$.
\end{cor}
\subsection{General case}
We sketch out the argument for computing $G$-projective covers of arbitrary $G$-flows as shown in \cite{BallHagler1996}. 

Let $X$ be a $G$-flow and let $Z$ be a set of representatives of the orbit classes in $X$ equipped with the discrete topology. Consider the space $G\times Z$ and define a continuous $G$-action via
\[
g\cdot (h,z) := (gh,z)
\]
for all $g\in G, (h,z) \in G\times Z$. Define $\eta\colon  G\times Z \to X$ via $\eta(g,z) := gz$. Notice by construction, $\eta$ is a uniformly continuous surjective $G$-equivariant map and thus extends to an extension $\phi\colon \Sa(G\times Z) \to X$. Let $Y \subseteq \Sa(G\times Z)$ be a subflow minimal with respect to $\phi(Y ) = X$. We claim that $Y \simeq \proj_G(X)$. Let $C$ be a $G$-flow and $\pi\colon  C \to X$ a $G$-irreducible extension. It suffices to show there exists a $G$-map $\psi\colon  Y \to C$ such that $\pi\circ \psi = \phi|_Y$. For each $z\in Z$ fix some $y_z \in \pi^{-1}(\set{z})$. Now define $\xi\colon  G\times Z \to Y$ via
\[
\xi (g,z) = gy_z.
\]
Now $\xi$ is a uniformly continuous $G$-map and hence extends to a $G$-map $\Psi\colon  \Sa(G\times Z) \to Y$. By checking on the dense subset $G\times Z$ we obtain that $\pi \circ \Psi = \phi$. Thus $\psi := \Psi|_Y$ is our desired map. 
\begin{prop}\label{Prop:GxZProj}
    If $Z$ is a discrete set, then $\Sa(G\times Z)$ is a $G$-projective flow.
\end{prop}
\begin{proof}
    Suppose $X$ and $Y$ are $G$-flows, $\phi\colon  \Sa(G\times Z) \to X$ a $G$-map, and $\pi\colon Y \to X$ a surjective $G$-map. For each $z \in Z$, fix some $y_z \in \pi^{-1}(\set{\phi(1_G,z_y)})$. Now define the map $\eta\colon  G\times Z \to Y$ via $\eta(g,z) := gy_z$. Then $\eta$ is uniformly continuous and thus extends to a $G$-map $\psi\colon  \Sa(G\times Z) \to Y$. By checking on the dense subset $G\times Z$ we obtain that $\pi \circ \psi = \phi$.
\end{proof}
\begin{prop}
Let $X$ be a $G$-flow. Then $X$ is a $G$-projective flow if and only if $X$ is (up to isomorphism) a retract of $\Sa(G\times Z)$ for some discrete set $Z$.
\end{prop}
\begin{proof}
   $(\Rightarrow)$ By the discussion above, there exists a discrete set $Z$ and an extension $\phi\colon \Sa(G\times Z) \to X$. Now passing to a subflow $Y \subseteq \Sa(G\times Z)$ minimal with respect to $\phi(Y) = X$, we obtain that $\psi:= \phi|_Y\colon  Y\to X$ is an isomorphism. Now $\psi^{-1} \circ\phi\colon  \Sa(G\times Z) \to Y$ is a retract of $Y$. 
   
    $(\Leftarrow)$ Follows by Proposition \ref{Prop:GxZProj} and Corollary \ref{Cor:RetractProj}.
\end{proof}
\section{Affine \texorpdfstring{$G$}{G}-projective flows}

\begin{defn}\label{Defn:AffineProj}
    An affine $G$-flow $K$ is \emph{affine $G$-projective} if for all affine $G$-flows $L$ and $P$, affine $G$-maps $\phi\colon X \to L$, and affine extensions $\pi\colon  P \to L$, there exists an affine $G$-map $\psi\colon K \to P$ such that $\pi \circ \psi = \phi$.
\end{defn}
Kadison's representation theorem gives us the following.
\begin{cor}
    An affine $G$-flow $K$ is affine $G$-projective if and only if $A(K)$ is injective in the category of $G$-function systems with respect to $G$-equivariant unital positive maps. 
\end{cor}
\begin{remark}
    We can actually strength from injectivity in the category of $G$-function systems to the category of all \emph{$G$-operator systems} with respect to $G$-equivariant \emph{unital completely positive maps} (see \cite{Paulsen2002} for the definitions). The following is a brief sketch of the argument.
    
    Suppose the $G$-function system $A(K)$ is injective in the category of $G$-function systems. Let $E_1 \subseteq E_2$ be $G$-operator systems and $\phi\colon E_1 \to A(K)$ be a $G$-equivariant unital completely positive map. Now by Kadison's representation theorem each $E_i$ is \emph{unital order isomorphic} (not neccesairly completely) to the affine functions on their state space, $A(S(E_i))$. Thus $\phi$ extends to a $G$-equivariant unital positive map $\psi\colon E_2 \to A(K)$. Now any unital positive map whose codomain is commutative must actually be completely positive (see \cite[Theorem 3.9]{Paulsen2002}). Thus $\phi$ extends to a unital completely positive map $\psi$ and hence $A(K)$ is injective in the category of $G$-operator systems.
\end{remark}

\begin{prop}\label{G-proj-implies-affine-G-proj-proposition}
    If $X$ is a $G$-projective flow, then $P(X)$ is an affine $G$-projective flow.
\end{prop}
\begin{proof}
    Suppose $K$ and $L$ are affine $G$-flows, $\phi\colon P(X) \to K$ an affine $G$-map, and $\pi\colon  L \to K$ an affine extension. As $\phi|_X\colon X \to K$ is a $G$-map and $X$ is $G$-projective, there exists a $G$-map $\Tilde{\psi}\colon  X \to L$ such that $\pi \circ \Tilde{\psi} = \phi|_X$. Let $\psi\colon  P(X) \to P(L)$ be the affine $G$-map induced by $\Tilde{\psi}$ and $\beta\colon  P(L) \to L$ denote the barycenter map. It follows that $\eta := \beta \circ \psi\colon P(X) \to L$ is an affine $G$-map such that for all $x \in X$
    \[
    \pi(\eta(\delta_x)) = \pi(\beta(\delta_{\Tilde{\psi}(x)})) = \pi(\Tilde{\psi}(x)) = \phi(\delta_x).
    \]
    As the affine $G$-maps $\pi \circ \eta$ and $\phi$ agree on the extreme points of $P(X)$, we obtain that $\pi \circ \eta = \phi$.
\end{proof}
\begin{example}
    The above proposition tells us that $P(\Sa(G))$ and $P(M(G))$ are always affine $G$-projective flows. Moreover, by Theorem \ref{Thm:ExAmenSubgp} $P(\Sa(G/H))$ is an affine $G$-projective flow whenever $G$ is a Polish group and $H$ is an extremely amenable closed subgroup of $G$. Actually, we can obtain an exact characterization of when an affine flow of the form $P(\Sa(G/H))$ is $G$-projective.
\end{example}

\begin{theorem} \label{Thm:AmenSubgp}
   If $G$ is a Polish group and $H \leq G$ is a closed subgroup, then $H$ is amenable if and only if $P(\Sa(G/H))$ is affine $G$-projective. 
\end{theorem}
\begin{proof}
    $(\Rightarrow)$ Suppose $K$ and $L$ are affine $G$-flows, $\phi\colon P(\Sa(G/H)) \to K$ an affine $G$-map, and $\pi\colon  L \to K$ an affine extension. Let $x := \phi(\delta_H)\in K$ be the image of the trivial coset. Notice that $\pi^{-1}(\set{x}) \subseteq L$ is a nonempty closed convex subset. As $H \subseteq \Stab(\set{x})$ we obtain that $\pi^{-1}(\set{x})$ is an $H$-invariant subset and hence an affine $H$-flow. Therefore, there exists an $H$-fixed point $y \in \pi^{-1}(\set{x})$. Now consider the $G$-map $\Tilde{\psi}\colon \Sa(G/H)\to L$ given via $\Tilde{\psi}(gH) = gy$. This incudes an affine $G$-map $P(\Sa(G/H)) \to P(L)$, which upon composing with the barycenter map $P(L) \to L$ yields an affine $G$-map $\psi\colon  P(\Sa(G/H)) \to L$. Now given any $gH\in G/H$ we have that
    \[
    \pi(\psi(\delta_{gH})) = \pi(gy) = \phi(gH).
    \]
    By continuity $\pi\circ \psi$ and $\phi$ agree on $\Sa(G/H) \subseteq P(\Sa(G/H))$. Thus we must have $\pi \circ \psi = \phi$. 
    
    $(\Leftarrow)$ It suffices to show $P(\Sa(H))$ contains an $H$-fixed point. Consider the extension $\phi\colon \Sa(G) \to \Sa(G/H)$ given by $\phi(g) = gH$. This induces an affine extension $\eta\colon P(\Sa(G)) \to P(\Sa(G/H))$. We claim that $\eta^{-1}(\set{\delta_H}) = P(\Sa(H))$. Notice that as $\delta_H \in P(\Sa(G/H))$ is an extreme point we obtain that $\eta^{-1}(\set{\delta_H})$ is a nonempty closed face in $P(\Sa(G/H))$. Thus it follows that $\ext(\eta^{-1}(\set{\delta_H})) \subseteq \ext (P(\Sa(G))) = \Sa(G)$. By Lemma \ref{Lem:Stab} we have that $\phi^{-1}(\set{H}) = \Sa(H)$. As $\eta$ is induced by $\phi$ we obtain that
    \[
    \ext(\eta^{-1}(\set{\delta_H})) \subseteq \Sa(G) \cap \phi^{-1}(\set{H}) = \Sa(H).
    \]
    Thus $\eta^{-1}(\set{\delta_H}) \subseteq P(\Sa(H))$ and the reverse inclusion is clear. Now by affine $G$-projectivity there exists an affine $G$-map $\psi\colon  P(\Sa(G/H)) \to P(\Sa(G))$ such that $\eta \circ \psi = \text{id}$. It follows that $\psi(\delta_H) \in \eta^{-1}(\set{\delta_H}) = P(\Sa(H))$ is an $H$-fixed point. 
\end{proof}
\begin{cor}\label{Cor:AmenSubgpInj}
    If $G$ is a Polish group and $H \leq G$ is a closed subgroup, then $H$ is amenable if and only if $\operatorname{RUC}^b(G/H)$ is $G$-injective with respect to $G$-equivariant unital positive maps. 
\end{cor}
\begin{proof}
    Follows by Theorem \ref{Thm:AmenSubgp} and the observation that $\operatorname{RUC}^b(G/H)$ is isomorphic (as $G$-function systems) to $A(P(\Sa(G/H)))$.
\end{proof}

    While not explicitly mentioned, the above corollary in the case of countable discrete groups follows from \cite[Corollary 4.4.5]{Anantharaman-Delaroche2003} (also see \cite[Theorem 6]{CapraceMonod2014}). In particular, they show that for a locally compact group $G$, a closed subgroup $H \leq G$ is amenable if and only if there exists a $G$-equivariant condition expectation from $L^\infty(G)$ to $L^\infty(G/H)$. When $G$ is a discrete group we see that 
    \[
    L^\infty(G) = \ell^\infty(G) = \operatorname{RUC}^b(G).
    \]
    and similarly $L^\infty(G/H) = \operatorname{RUC}^b(G/H)$. A basic fact from operator algebras is that given an injective $C^*$-algebra $A$ and a $C^*$-subalgebra $B \subseteq A$, we have that $B$ is injective if and only if there exists a conditional expectation $\Phi\colon A \to B$ (see \cite{Paulsen2002}). Moreover, this fact remain true if we replace $A,B$ with $G$-$C^*$-algebras and require the condition expectation to be $G$-equivariant. Thus Corollary \ref{Cor:AmenSubgpInj} for countable discrete groups is seen to be equivalent to \cite[Corollary 4.4.5]{Anantharaman-Delaroche2003}. Moreover, we have the following.
    \begin{cor}
        Let $G$ be a Polish locally compact group and $H \leq G$ a closed subgroup. Then the following conditions are equivalent:
        \begin{enumerate}
            \item $H$ is amenable.
            \item $\operatorname{RUC}^b(G/H)$ is $G$-injective.
            \item There exists a $G$-equivariant conditional expectation $\Phi\colon  \operatorname{RUC}^b(G) \to \operatorname{RUC}^b(G/H)$.
            \item There exists a $G$-equivariant conditional exception $\Phi\colon  L^\infty(G) \to L^\infty(G/H)$.
        \end{enumerate}
    \end{cor}
    It would be interesting to see if there is a direct proof of the equivalence of conditions $(3)$ and $(4)$ without passing through the amenability of $H$.
\subsection{Affine $G$-projective covers}
As we saw before, the property of $G$-projectivity can be characterized with being maximal with respect to certain extensions. In particular, $X$ is $G$-projective if and only if $X$ has no nontrivial $G$-irreducible extensions. We now show affine $G$-projectivity can be characterized with an analogous notion of extensions.
\begin{defn}
    Let $K,L$ be affine $G$-flows. We say an affine extension $\phi\colon K \to L$ is \emph{affine $G$-irreducible} if for all proper affine subflows $P \subseteq K$ we have $\phi(P) \neq L$.
\end{defn}
\begin{lemma}\label{Lem:AffIrrImpliesIrr}
    Let $K$ and $L$ be affine $G$-flows. If $\phi\colon L \to K$ is an affine $G$-irreducible extension, then $\phi$ restricts to a $G$-irreducible extension from $\overline{\ext}(L)$ to $\overline{\ext}(K)$.
\end{lemma}
\begin{proof}
    We first claim $\phi(\overline{\ext}(L)) = \overline{\ext}(K)$. Notice for each $x \in \ext(K)$, $\phi^{-1}(\set{x})$ is a face in $L$ and thus $$\phi^{-1}(\set{x}) \cap \overline{\ext}(L) \neq \emptyset.$$
    Let $C := \phi^{-1}(\overline{\ext}(K)) \cap \overline{\ext}(L)$ and notice that $\overline{\conv}(C)$ is an affine subflow such that $\phi(\overline{\conv}(C)) = K$. By Milman's theorem and the assumption that $\phi$ is affine $G$-irreducible we see that
    \[
    \overline{\ext}(L) \subseteq C.
    \]
    The claim follows and a similar argument using Milman's theorem shows $\phi\colon \overline{\ext}(L) \to \overline{\ext}(K)$ is a $G$-irreducible extension. 
\end{proof}

\begin{lemma}\label{Lem:AffProjAffIrr} 
    If $K$ is an affine $G$-flow, then $K$ is affine $G$-projective if and only if $K$ has no nontrivial affine $G$-irreducible extensions.
\end{lemma}
\begin{proof}
    $(\Rightarrow)$ Let $L$ be an affine $G$-flow and $\phi\colon L \to K$ an affine $G$-irreducible extension. As $K$ is affine $G$-projective, there exists an affine $G$-map $\psi\colon K \to L$ such that $\phi \circ \psi = \text{id}_K$. Using that $\phi$ is affine $G$-irreducible we obtain that $\psi$ is surjective and thus showing $\phi$ is injective. 
    
    $(\Leftarrow)$ Let $L,P$ be affine $G$-flows, $\phi\colon K \to L$ an affine $G$-map, and $\pi\colon P \to L$ an affine extension. Now define the affine subflow $W \subseteq K\times P$ via
    \[
    W:= \set{(x,y)\in K\times P: \phi(x) = \psi(y)}.
    \]
    It follows that the projection onto the first coordinate $\rho\colon W \to K$ is an affine extension. By Zorn's lemma, let $W' \subseteq W$ be an affine subflow minimal with respect to $\rho(W') = K$. By assumption, $\rho|_{W'}\colon  W' \to K$ is an isomorphism. Let $\eta\colon  W \to L$ denote the projection onto the second coordinate. Then $\psi:= \eta \circ (\rho|_{W'})^{-1}\colon K \to L$ is the required affine $G$-map which satisfies $\pi \circ \psi = \phi$.
\end{proof}
\begin{prop}\label{Prop:StrProxExt}
    Let $X$ and $Y$ be $G$-flows where $Y$ is also assumed to be minimal. Then there exists an affine $G$-irreducible extensions $\psi\colon P(X) \to P(Y)$ if and only if there exists a $G$-irreducible strongly proximal extension $\phi\colon X \to Y$.
\end{prop}
\begin{proof}
    $(\Rightarrow)$ By Lemma \ref{Lem:AffIrrImpliesIrr} we obtain that $\phi:= \psi|_X\colon X \to Y$ is a $G$-irreducible extension. Now let $y\in Y$ and $\mu \in P(\phi^{-1}(\set{y})) \subseteq P(X)$ be a probability measures supported on the fiber of $y$. As $Y$ is minimal we see that
    \[
    \psi(\overline{\conv}(\overline{G\cdot \mu})) = P(Y).
    \]
    Therefore, $\overline{\conv}(\overline{G\cdot \mu}) = P(X)$ and Milman's theorem tells us that $X \subseteq \overline{G\cdot \mu}$. It follows that $\phi$ is a strongly proximal extension. 
    
    $(\Leftarrow)$ Let $\phi\colon X \to Y$ be a $G$-irreducible strongly proximal extension. We claim the induced affine extension $\psi\colon P(X) \to P(Y)$ is affine $G$-irreducible. Let $Z \subseteq P(X)$ be an affine subflow such that $\psi(Z) = P(Y)$. Fix $y \in Y$ and consider $$\mu \in Z \cap \psi^{-1}(\set{y}) = Z \cap P(\phi^{-1}(\set{y})).$$ By strong proximality of $\phi$ we obtain that there exists $p \in \Sa(G)$ and $x \in X$ such that $\delta_x = p\mu \in Z$. Since $\phi$ is $G$-irreducible we obtain that $X$ is a minimal $G$-flow and thus $X \subseteq Z$. It follows that $Z = P(X)$, showing $\psi$ is an affine $G$-irreducible extension.
\end{proof}
\begin{remark}
    In the above theorem as $Y$ was assumed to be minimal we have that $\phi\colon X \to Y$ is $G$-irreducible if and only if $X$ is minimal. In particular, if $X$ and $Y$ are minimal $G$-flows, then $X$ is a strongly proximal extension of $Y$ if and only if $P(X)$ is an affine $G$-irreducible extension of $P(Y)$. 
\end{remark}
Using Lemma \ref{Lem:AffProjAffIrr} as motivation we have the following definition.
\begin{defn}
    Let $K$ be an affine $G$-flow. We say the pair $(L,\psi)$ is an \emph{affine $G$-projective cover} of $K$ if $L$ is an affine $G$-projective flow and $\phi\colon L \to K$ is an affine $G$-irreducible extension. 
\end{defn}
To prove that affine $G$-projective cover exist it suffices, by Lemma \ref{Lem:AffProjAffIrr}, to show that for every affine $G$-flow $K$ there exists an affine $G$-irreducible extension $\psi\colon L \to K$ such that $L$ itself has no nontrivial affine $G$-irreducible extensions. To prove this we essentially repeat the exact same argument in \cite{BallHagler1996} used to prove the existence of $G$-projective covers. 
\begin{theorem}\label{Thm:AffProjCoverExist}
    Let $K$ be an affine $G$-flow. Then there exists an affine $G$-flow $L$ and an affine $G$-irreducible extension $\phi\colon L \to K$ such that $L$ itself has no nontrivial affine $G$-irreducible extensions. Moreover, $L$ is up to isomorphism the unique such affine $G$-flow.
\end{theorem}
\begin{proof}
    Towards a contradiction suppose $K$ has no such affine $G$-irreducible extensions. We claim for each $ \lambda \in \operatorname{ORD}$ there exists an affine $G$-flow $L_\lambda$ along with a nontrivial affine $G$-irreducible extension $\phi_\lambda\colon  L_\lambda \to K$. Moreover, if $\beta < \lambda$, then there exists an affine $G$-irreducible extension $\gamma_\lambda^\beta\colon  L_\lambda \to L_\beta$ such that $$\phi_\beta \circ \gamma_\lambda^\beta  = \phi_\lambda$$and $\gamma_\beta^\alpha \circ \gamma_\lambda^\beta = \gamma_\lambda^\alpha$ for all $\alpha < \beta < \lambda$. 
    
    By assumption there exists a nontrivial affine $G$-irreducible $\phi\colon  L \to K$. Define $L_0:= L$ and $\phi_0:= \phi$. Now suppose the claim holds for all $\beta \leq \lambda$. By assumption $L_\lambda$ has a nontrivial affine $G$-irreducible. Let $L_{\lambda + 1}$ and $\gamma_{\lambda+1}^\lambda$ be said extension. Now define 
    \[
    \gamma_{\lambda+1}^\beta := \gamma_{\lambda+1}^\lambda \circ \gamma_{\lambda}^\beta
    \]
    for all $\beta < \lambda$. Lastly, define $\phi_{\lambda+ 1} := \phi_{\lambda} \circ \gamma_{\lambda+1}^\lambda$. As the composition of two affine $G$-irreducible extension remains affine $G$-irreducible the claim also also holds with $\lambda +1$. Now suppose $\lambda$ is a limit ordinal and the claim holds for all ordinals $\beta < \lambda$. Let $L_{\lambda} := \varprojlim L_\beta$ be the inverse limit, i.e. 
    \[
    L_{\lambda} := \set{(x_\beta)_{\beta < \lambda} \in \prod_{\beta < \lambda} L_\beta : \gamma_{\beta}^\alpha(x_\beta) = x_\alpha \text{ for all } \alpha < \beta < \lambda}.
    \]
    It can be checked that $L_\lambda$ is an affine $G$-flow. Now define $\gamma_\lambda^\beta\colon  L_\lambda \to L_\beta$ to the the projection map and observe that $\phi_{\lambda} := \phi_\beta \circ \gamma_{\lambda}^\beta$ is independent of the choice of $\beta < \lambda$. Notice that $\phi_{\lambda}\colon L_\lambda \to K$ is an affine $G$-irreducible extension. Indeed, if $C \subseteq L_\lambda$ is an affine subflow such that $\phi_{\lambda}(C) = K$ it follows that $\phi_\beta(\gamma_\lambda^\beta(C)) = K$ for all $\beta < \lambda$. Since $\phi_\beta$ is an affine $G$-irreducible extension, we obtain that $\gamma_{\lambda}^\beta (C)= L_\beta$ for all $\beta < \lambda$ and thus $C = L_\lambda$ by the universal property of the inverse limit. One can show in a similar fashion that each $\gamma_\lambda^\beta$ is also an affine $G$-irreducible extension and that all the maps are compatible as desired. Thus by transfinite induction the claim holds true. 
    
    Now if $\phi\colon  L \to K$ is affine $G$-irreducible extension, then the cardinality of $L$ is bounded above as follows. For each $x \in {\ext}(K)$ pick $y_x \in \ext(L)$ such that $\phi(y_x)= x$. It follows that
    \[
    \phi(\overline{\conv}(G\cdot \set{y_x: x\in \ext(K)})) = K
    \]
    and thus $L = \overline{\conv}(G\cdot \set{y_x: x\in \ext(K)})$. Therefore we obtain that
    \[
    |L| \leq 2^{2^{|\R||G||K|}}.
    \]
    However, one can show that $|L_\lambda| \geq |\lambda|$ in the above construction as each extension was assumed to be nontrivial (see \cite[Theorem 3.8]{BallHagler1996}), a contradiction to the above bound. Therefore, there must exists an affine $G$-irreducible extension $\phi\colon L \to K$ such that $L$ itself has no nontrivial affine $G$-irreducible extensions. 
    
    Uniqueness of $L$ follows from the exact same argument as shown in the converse direction of Lemma \ref{Lem:AffProjAffIrr}. Indeed, if $\pi\colon P \to K$ is also an affine $G$-irreducible extension, then by the proof of Lemma \ref{Lem:AffProjAffIrr} there exists an affine $G$-map $\psi\colon  L \to P$ such that $\pi \circ \psi = \phi$. Using that $\pi$ and $\phi$ are affine $G$-irreducible, it is not hard to see that $\psi$ must also be affine $G$-irreducible. Therefore, if $P$ was assumed to have no nontrivial affine $G$-irreducible extensions, then $\psi$ must be an isomorphism.
\end{proof}
\begin{lemma}\label{Lem:BarycenterAffIrr}
    If $K$ is an affine $G$-flow, then the barycenter map $$\beta\colon  P(\overline{\ext}(K)) \to K$$ is an affine $G$-irreducible extension. 
\end{lemma}
\begin{proof}
    First by \cite[Proposition 1.2]{Phelps2001}, we see that $\beta$ is a surjective map. By \cite[Proposition 1.4]{Phelps2001}, any extreme point $x \in \ext(K)$ has a unique representing measure on $K$ given by its Dirac mass $\delta_x$, i.e., $\beta^{-1}(\set{x}) = \set{\delta_x}$. Hence, if $C \subseteq P(\overline{\ext}(K))$ is an affine subflow such that $\beta(C) = K$, then we must have that $\delta_x \in C$ for all $x \in \ext(K)$. This forces $C = P(\overline{\ext}(K))$.
\end{proof}
\begin{cor}\label{Cor:AffProjFlows}
    If $K$ is an affine $G$-projective flow, then $K \simeq P(X)$ for some $G$-flow $X$.
\end{cor}
\begin{proof}
    By Lemma \ref{Lem:BarycenterAffIrr}, the barycenter map $\beta\colon  P(\overline{\ext}(K)) \to K$ is an affine $G$-irreducible extension. Now Lemma \ref{Lem:AffProjAffIrr} implies that $\beta$ must be an isomorphism.
\end{proof}
\begin{remark}
    Another way to prove the above is as follows. Notice that $A(K) \subseteq C(K)$ is a function system. Extending the identity map $A(K) \to A(K)$ yield a $G$-equivariant conditional expectation $\phi\colon C(K) \to A(X)$. Now the Choi-Effros product on $A(K)$ defined by
    \[
    f\ast g := \phi(f\cdot g)
    \]
    is a $C^\ast$-product on $A(K)$. In particular, $A(K)$ with the product $\ast$ is a commutative unital $C^\ast$-algebra (see \cite{ChoiEffros1977}). Thus $A(K)$ is isomorphic (as $G$-function systems) to $C(X)$ for some $G$-flow $X$. By Kadison's representation theorem, $K \simeq P(X)$.
\end{remark}
Thus we see that affine $G$-projectivity is a property of $G$-flows. In particular, without any loss of generality Definition \ref{Defn:AffineProj} could have only been stated for $G$-flows.

Now an application of Proposition \ref{Prop:StrProxExt} and the existence of affine $G$-projective covers shows that every minimal $G$-flow has a unique largest minimal strongly proximal extension.
\begin{prop}\label{Prop:MaxStrProxExt}
    Let $Y$ be a minimal $G$-flow. Then there exists a minimal $G$-flow $X$ and a strongly proximal extension $\phi\colon X \to Y$ such that if $Z$ is a minimal $G$-flow and $\psi\colon Z \to Y$ is a strongly proximal extension, then there exists a strongly proximal extension $\pi\colon X \to Z$ such that $\psi\circ\pi = \phi$. Moreover, $X$ is unique up to isomorphism.
\end{prop}
\begin{proof}
    Let $\psi\colon P(X) \to P(Y)$ be the affine $G$-projective cover of $P(Y)$. By Proposition \ref{Prop:StrProxExt}, this induces a strongly proximal extension $\phi\colon X \to Y$. The remaining desired properties of $X$ also follow from the Proposition \ref{Prop:StrProxExt} and the fact that $P(X)$ is the affine $G$-projective cover of $P(Y)$.
\end{proof}
\begin{example}
It is not hard to see that $P(\Pi_s(G))$ is the affine $G$-projective cover of the trivial flow $\set{*}$. Indeed, let $\phi\colon P(X) \to \set{*}$ be the affine $G$-projective cover of $\set{*}$. By Proposition \ref{Prop:StrProxExt}, $X$ is a minimal strongly proximal $G$-flow and thus there exists a $G$-map $\psi\colon  \Pi_s(G) \to X$. Once again Proposition \ref{Prop:StrProxExt} implies that the induced affine map $\tilde{\psi}\colon  P(\Pi_s(G)) \to P(X)$ is an affine $G$-irreducible extension and thus an isomorphism. 
\end{example}
A well known construction of $\Pi_s(G)$ is as follows. Let $K \subseteq P(M(G))$ be a minimal affine subflow. Then we must have that $K \simeq P(\Pi_s(G))$ (see \cite{Glasner1976}). Recall that $M(G)$ is the $G$-projective cover of $\set{*}$ and as shown above $P(\Pi_s(G))$ is the affine $G$-projective cover of $\set{*}$. This turns out to be a special case of the following general result.
\begin{theorem}\label{Thm:AltConOfAffProjCover}
    Let $K$ be an affine $G$-flow and $X$ be the $G$-projective cover of $\overline{\ext}(K)$. Then there exists an affine extension $\phi\colon  P(X) \to K$ and if $L \subseteq P(X)$ is an affine subflow minimal with respect to $\phi(L) = K$ we have that $L$ is isomorphic to the affine $G$-projective cover of $K$.
\end{theorem}
\begin{proof}
  Let $\beta_1\colon  P(\overline{\ext}(K)) \to K$ be the barycenter map which is an affine $G$-irreducible extension via Lemma \ref{Lem:BarycenterAffIrr}. Now let $X := \operatorname{Proj}_G(\overline{\ext}(K))$ be the $G$-projective cover and $\tilde{\phi}\colon  X \to \overline{\ext}(K)$ the given $G$-irreducible extension. We can then induce an affine extension from $\tilde{\phi}$ which upon composing with the barycenter map yields an affine extension 
  \[
  \phi\colon  P(X) \to K.
  \]
  By Zorn's lemma, let $L \subseteq P(X)$ be an affine subflow minimal with respect to $\phi(L) = K$. We claim that $(L, \phi|_Y)$ is the affine $G$-projective cover of $K$. Indeed, by construction $\phi|_L\colon  L \to K$ is an affine $G$-irreducible extension. Now by Theorem \ref{Thm:AffProjCoverExist} let $\pi\colon P(Y) \to  K$ be the affine $G$-projective cover of $K$. By Lemma \ref{Lem:AffIrrImpliesIrr}, $\pi$ restricts to a $G$-irreducible extension from $Y$ to $\overline{\ext}(K)$. Thus there exists a $G$-map $\rho\colon  X \to Y$ such that $\phi = \pi \circ \rho$. Taking the affine extension of $\rho$ yields an affine extension
  \[
  \tilde{\rho}\colon P(X) \to P(Y).
  \]
  One can show that $\phi = \pi \circ \tilde{\rho}$ holds and thus  $\tilde{\rho}|_L\colon  L \to P(Y)$ is an affine $G$-irreducible extension. Therefore, $\tilde{\rho}|_L$ is an isomorphism. The above argument can be summarized by the following diagrams:
\[\begin{tikzcd}
	& Y &&&& {P(Y)} \\
	X & {\overline{\ext}(K)} && {L \subseteq P(X)} & {P(\overline{\ext}(K))} & K
	\arrow["{\pi|_Y}", from=1-2, to=2-2]
	\arrow["\pi", from=1-6, to=2-6]
	\arrow["\rho", dotted, from=2-1, to=1-2]
	\arrow[from=2-1, to=2-2]
	\arrow["{\tilde{\rho}}", dotted, from=2-4, to=1-6]
	\arrow[from=2-4, to=2-5]
	\arrow[from=2-5, to=2-6]
\end{tikzcd}\]
\end{proof}
The following result was shown in \cite{KennedyRaumSalomon2022} for discrete groups. 
\begin{prop}
   $P(\Pi(G))$ is an affine $G$-projective flow.
\end{prop}
\begin{proof}
    By Theorem \ref{Thm:AffProjCoverExist}, let $\phi\colon  P(X) \to P(\Pi(G))$ be the affine $G$-projective cover of $P(\Pi(G))$. Now Proposition \ref{Prop:StrProxExt} implies that $\phi|_X\colon  X \to \Pi(G)$ is a strongly proximal $G$-irreducible extension. In particular, $X$ is a minimal $G$-flow and $\phi|_X$ is a proximal extension. Therefore, $X$ must be a minimal proximal $G$-flow and the universal property of $\Pi(G)$ implies that $\phi|_X$ must be an isomorphism. Hence, $\phi$ is an isomorphism as well.  
\end{proof}
We now prove an analogue of Theorem \ref{Thm:ProjCoverPointTrans} for affine $G$-projective covers. Recall the following structure result for the Furstenberg boundary for Polish groups proven in \cite[Theorem 7.5]{Zucker2021}.
\begin{theorem}[\cite{Zucker2021}]
    Let $G$ be a Polish group. If $\Pi_s(G)$ has a comeager orbit, then $\Pi_s(G) \simeq \Sa(G/H)$ for some maximal amenable closed subgroup $H \leq G$.
\end{theorem}
\begin{theorem}\label{Thm:AffProjCoverPointTrans}
    Let $G$ be a Polish group and $H \leq G$ a closed subgroup. If $\Pi_s(H)$ has a comeager orbit, then $P(\Sa(G/H^*))$ is isomorphism to the affine $G$-projective cover of $P(\Sa(G/H))$ where $H^* \leq H$ is a maximal amenable subgroup such that $\Pi_s(H) \simeq \Sa(H/H^*)$.
\end{theorem}
\begin{proof}
    Consider the canonical extension $\phi\colon \Sa(G/H^*) \to \Sa(G/H)$ given via $gH^* \mapsto gH$. Let $\psi\colon P(\Sa(G/H^*))\to P(\Sa(G/H))$ be the induced affine extension. By Theorem \ref{Thm:AmenSubgp}, $P(\Sa(G/H^*))$ is an affine $G$-projective flow. Thus it suffices to show $\psi$ is an affine $G$-irreducible extension. As shown in the proof of Theorem \ref{Thm:ProjCoverPointTrans} we have that $\phi^{-1}(\set{H})= \Sa(H/H^*)$ and thus
    \[
    \psi^{-1}(\set{\delta_H}) = P(\Sa(H/H^*)).
    \]
    Now suppose $Z \subseteq P(\Sa(G/H^*))$ is an affine subflow such that $\psi(Z) = P(\Sa(G/H))$. It follows that 
    \[
   \emptyset \neq  Z \cap \psi^{-1}(\set{\delta_H}) = Z \cap P(\Sa(H/H^*))
    \]
    is nonempty. Using strongly proximality and minimality of $\Sa(H/H^*)$ we see that $\delta_{H^*} \in Z$ and thus $Z= P(\Sa(G/H^*))$.
\end{proof}
Suppose $X$ is a $G$-flow such that $x_0 \in X$ has comeager orbit. If the Furstenberg boundary of $\Stab(x_0)$ contains a comeager orbit, then the above theorem allows us to compute the affine $G$-projective cover of $P(X)$. Even if we don't make any assumptions on $\Pi_s(\Stab(x_0))$ we can still compute the affine $G$-projective cover of $P(X)$. Indeed, it will be isomorphic to $\overline{\conv}(\overline{G\cdot\Pi_s(\Stab(x_0))})$ when viewing $\Pi_s(\Stab(x_0)) \subseteq P(\Sa(G))$ as a closed subset. We delay proving this until introducing more machinery and techniques in the next section (see Proposition \ref{Prop:AffProjCoverPointTransGen}). 
\section{Proximally irreducible flows}
Motivated by the property of algebraically irreducible affine $G$-flows in \cite{KennedyRaumSalomon2022} we define the following notion of affine extensions.
\begin{defn} \label{Def:ProxIrr}
    Let $L$ and $K$ be affine $G$-flows. An affine extension $\phi\colon L \to K$ is \emph{proximally irreducible} if for any affine subflow $C \subseteq L$ satisfying 
   \[
 C \cap \conv (\overline{\ext}(L)\cap \phi^{-1}(\set{y})) \neq \emptyset, \quad \text{for all } y\in \overline{\ext}(K) 
    \]
  we must have that $C = L$. 
\end{defn}
With this it is not hard to see that the minimal algebraically irreducible flows as defined in \cite{KennedyRaumSalomon2022} are exactly the proximally irreducible extensions of the trivial $G$-flow.
\begin{defn}
    An affine $G$-flow $K$ is said to be \emph{proximally irreducible} if $K$ has no nontrivial proximally irreducible extensions. 
\end{defn}
\begin{example}
    Notice if $K$ is a proximally irreducible $G$-flow, then $K$ must be affine $G$-projective. Indeed, if $L$ is an affine $G$-flow and $\phi\colon L \to K$ is an affine $G$-irreducible extension, then $\phi$ is also a proximally irreducible extension. Therefore, $\phi$ is an isomorphism and thus by Lemma \ref{Lem:AffProjAffIrr}, $K$ is affine $G$-projective. In particular, by corollary \ref{Cor:AffProjFlows} we have that $K \simeq P(X)$ for some $G$-flow $X$.
\end{example} 
With basically the exact same proof of Lemma \ref{Lem:AffIrrImpliesIrr} we have the following analogous result for proximally irreducible extensions. 
\begin{lemma}\label{Lem:ProxIrrImpliesIrr}
    Let $L$ and $K$ be affine $G$-flows. If $\phi\colon L \to K$ is a proximally irreducible extension, then $\phi$ restricted to $\overline{\ext}(L)$ is a $G$-irreducible extension from $\overline{\ext}(L)$ to $\overline{\ext}(K)$.
\end{lemma}
\begin{proof}
    Consider the $G$-flow defined via
    \[
    C := \overline{\ext}(L) \cap \phi^{-1}(\overline{\ext}(K)).
    \]
   Notice by construction and using the assumption that $\phi$ is proximally irreducible we obtain that $\overline{\conv}(C) = K$. Thus Milman's theorem implies that $\overline{\ext}(K) \subseteq L$ and hence $C= \overline{\ext}(L)$. Therefore $\phi|_{\overline{\ext}(L)}\colon \overline{\ext}(L) \to \overline{\ext}(K)$ is an extension of $G$-flows. A similar argument shows it must be a $G$-irreducible extension.
   \end{proof}
We now show an analogous result of Proposition \ref{Prop:StrProxExt} holds for proximally irreducible extensions.
\begin{prop}\label{Prop:ProxExt}
    Let $X$ and $Y$ be $G$-flow where $Y$ is also assumed to be minimal. Then there exists a $G$-irreducible proximal extension $\phi\colon  X \to Y$ if and only if there exists a proximally irreducible extension $\psi\colon  P(X) \to P(Y)$.
\end{prop}
\begin{proof}
    $(\Rightarrow)$ Suppose $\phi\colon  X\to Y$ is a $G$-irreducible proximal extension. It follows that $X$ must be a minimal $G$-flow. Let $\psi\colon  P(X) \to P(Y)$ be the induced affine extension and $Z \subseteq P(X)$ an affine subflow such that
    \[
     Z \cap \conv(X\cap\phi^{-1}(\set{y})) \neq \emptyset, \ \ \ \forall y\in Y.
    \]
    Given $y\in Y$ let $\lambda_1,\dots \lambda_n \in \R^+$ with $\sum_{j=1}^n \lambda_j = 1$ and $x_1,\dots, x_n\in \phi^{-1}(\set{y}) \subseteq X$ such that
    \[
    \lambda_1\delta_{x_1} + \cdots + \lambda_n \delta_{x_n} \in Z.
    \]
    As $\phi$ is a proximal extension, there exists $p \in \Sa(G)$ such that $px_1 = \cdots  = px_n$. It follows that
    \[
    p(\lambda_1\delta_{x_1} + \cdots + \lambda_n \delta_{x_n}) = \delta_{px_1} \in Z.
    \]
    By minimality of $X$, we obtain that $X \subseteq Z$ and hence $Z = P(X)$. 
    
    $(\Leftarrow)$ Suppose $\psi\colon  P(X) \to P(Y)$ is a proximally irreducible extension. By the previous lemma we have that $\phi: = \psi|_X\colon X \to Y$ is a $G$-irreducible extension. Now given any $y\in Y$ and $x_1,x_2 \in \phi^{-1}(y)$ we have that
    \[
    \psi\left(\frac{\delta_{x_1} + \delta_{x_2}}{2}\right) = \delta_y.
    \]
    By minimality of $Y$ it is not hard to check that that affine subflow $Z \subseteq X$ defined by
    \[
    Z := \overline{\conv} (\overline{G\cdot (\frac{\delta_{x_1} + \delta_{x_2}}{2})})
    \]
    satisfies 
    \[
       Z \cap \conv(X\cap\phi^{-1}(\set{y})) \neq \emptyset, \ \ \ \forall y\in Y.
    \]
    As $\psi$ is proximally irreducible we must have $Z = P(X)$ and thus by Milman's theorem $X \subseteq \overline{G\cdot (\frac{\delta_{x_1} + \delta_{x_2}}{2})}$. Fix $x \in X$ and pick $p\in \Sa(G)$ such that
    \[
    \delta_x = p(\frac{\delta_{x_1} + \delta_{x_2}}{2}) = \frac{\delta_{px_1} + \delta_{px_2}}{2}.
    \]
    Since $x$ is an extreme point we obtain that $px_1 = px_2$. Therefore, $\phi\colon X \to Y$ is a proximal extension.
\end{proof}
\begin{lemma}\label{Lem:CompAffIrrProxIrr}
    Suppose $K,L$, $P$ are affine $G$-flows. If $\phi\colon K \to P$ is an affine $G$-irreducible extension and $\psi\colon P \to L$ is proximally irreducible, then $\eta:= \psi \circ \phi\colon K \to L$ is proximally irreducible. 
\end{lemma}
\begin{proof}
    Let $C \subseteq K$ be an affine subflow such that 
    \[
    C \cap \conv (\overline{\ext}(K) \cap \eta^{-1}(\set{y})) \neq \emptyset, \ \ \ \forall y\in \overline{\ext}(L).
    \]
    To show $C = K$ it suffices to prove $\phi(C) = P$ as $\phi$ is an affine $G$-irreducible extension. Let $y\in \overline{\ext}(L)$ and consider 
    \[
    w:= \lambda_1x_1 + \cdots \lambda_n x_n \in C
    \]
    where the $\lambda_j$'s are positive real numbers with $\sum_{j=1}^n \lambda_j = 1$ and $x_1,\dots, x_n \in \overline{\ext}(K)$ are such that 
    \[
    y = \eta(x_j) = \psi(\phi(x_j)).
    \]
    As $\psi(\overline{\ext}(K)) = \overline{\ext}(P)$ we see that each $z_j:= \phi(x_j) \in \overline{\ext}(P)$ satisfies $z_j \in \psi^{-1}(\set{y})$. It follows that
    \[
    \phi(w) = \lambda_1z_1 + \cdots + \lambda_n z_n \in \phi(C) \cap \conv(\overline{\ext}(P) \cap \psi^{-1}(\set{y})).
    \]
    As this holds for all $y\in \overline{\ext}(L)$ and $\psi$ is proximally irreducible we obtain that $\phi(C) = P$. Therefore, $C = K$.
\end{proof}
\begin{example}
    We claim if $X$ is a $G$-projective flow, then $P(X)$ is proximally irreducible. Indeed, suppose $L$ is an affine $G$-flow and $\phi\colon L \to P(X)$ is a proximally irreducible extension. Let $\psi\colon  P(Y) \to L$ be the affine $G$-projective cover of $L$. It follows by Lemma \ref{Lem:CompAffIrrProxIrr} that $\eta = \phi \circ \psi\colon  P(Y) \to P(X)$ is a proximally irreducible extension. Thus $\eta|_Y\colon  Y \to X$ is a $G$-irreducible extension and thus an isomorphism. It follows that $\eta$ and consequently $\phi$ must be an isomorphism. Therefore, $P(X)$ is proximally irreducible. In particular, we see that $P(\Sa(G))$ and $P(M(G))$ are proximally irreducible. 
\end{example}
\begin{prop} \label{Prop:UMPFIsProxIrr}
    $P(\Pi(G))$ is a proximally irreducible affine $G$-flow.
\end{prop}
\begin{proof}
    Let $L$ be an affine $G$-flow and $\phi\colon L \to P(\Pi(G))$ a proximally irreducible extension. Now let $\psi\colon  P(X) \to L$ be the affine $G$-projective cover of $L$. By Lemma \ref{Lem:CompAffIrrProxIrr} we obtain that $\eta:= \phi \circ \psi \colon  P(X) \to P(\Pi(G))$ is a proximally irreducible extension. Thus Proposition \ref{Prop:ProxExt} tells us that $\eta|_X\colon  X \to \Pi(G)$ is a proximal extension and $X$ must be a minimal $G$-flow. In particular, $X$ must be a minimal proximal $G$-flow and thus $\eta|_X$ is an isomorphism by the universal property of $\Pi(G)$. It follows that $\eta$ and thus $\phi$ must be an isomorphism. 
\end{proof}
\subsection{Strong amenability of subgroups}
We now present a characterization of strong amenability for closed subgroups of Polish groups. This characterization answers an open question of Zucker \cite[Question 7.6]{Zucker2021}. We prove each direction separately for clarity. 
\begin{theorem}\label{Thm:StrAmenSubgp}
   If $G$ is a Polish group and $H \leq G$ is a closed subgroup, then $H$ is strongly amenable if and only if $P(\Sa(G/H))$ is proximally irreducible. 
\end{theorem}
\begin{proof}[Proof of Theorem \ref{Thm:StrAmenSubgp} $(\Rightarrow)$]
    Let $L$ be an affine subflow and consider a proximally irreducible extension $\phi\colon L \to P(\Sa(G/H))$. Let $X$ be a $G$-flow and $\psi\colon P(X) \to L$ an affine $G$-irreducible extension. By Lemma \ref{Lem:CompAffIrrProxIrr} we have that $\phi\circ \psi\colon P(X) \to P(\Sa(G/H))$ is proximally irreducible. To show $\phi$ is an isomorphism, it suffices to show $\eta:= \phi \circ \psi$ is an isomorphism. 
    
    We claim that $\eta^{-1}(\set{\delta_H}) \cap X$ is a minimal proximal $H$-flow. Fix $\delta_x\in \eta^{-1}(\set{\delta_H}) \cap X$ and consider the affine subflow 
    \[
    Z:= \overline{\conv}(\overline{G\delta_x}) \subseteq P(X).
    \]
   As $H \in \Sa(G/H)$ has dense orbit, we obtain that 
    \[
    Z \cap \conv(X \cap \eta^{-1}(\set{\delta_y})) \neq \emptyset, \ \ \ \forall y\in \Sa(G/H).
    \]
    Since $\eta$ is proximally irreducible we obtain that $Z = P(X)$. Milman's theorem gives us that $X \subseteq \overline{G\delta_x}$. In particular, for any $\delta_y \in \eta^{-1}(\set{\delta_H}) \cap X$ there exists $p \in \Sa(G)$ such that $p\delta_x = \delta_y$. Applying $\eta$ we obtain that
    \[
    p\delta_H = \delta_H
    \]
    and thus, by Lemma \ref{Lem:Stab}, we have that $p \in \Sa(H)$. Therefore, the $H$-orbit of any $\delta_x \in \eta^{-1}(\set{\delta_H})\cap X$ is dense in $\eta^{-1}(\set{\delta_H})\cap X$ and it follows that it is a minimal $H$-flow. The argument for proximality is similar. Indeed, let $\delta_{x_1}, \delta_{x_2} \in \eta^{-1}(\set{\delta_H})\cap X$ and consider the affine subflow
    \[
        \overline{\conv}(\overline{G\cdot (\frac{\delta_{x_1} + \delta_{x_2}}{2})}) \subseteq P(X).
    \]
    Using that $\eta$ is proximally irreducible and Milman's theorem we have that $X \subseteq  \overline{G\cdot (\frac{\delta_{x_1} + \delta_{x_2}}{2})}$. Thus given any $\delta_y\in \eta^{-1}(\set{\delta_H})\cap X$ there exists $p \in \Sa(G)$ with 
    \[
    \delta_y = \frac{p\delta_{x_1} + p \delta_{x_2}}{2}.
    \]
    Since $\delta_y$ is an extreme point we obtain that $p\delta_{x_i} = \delta_y$. As before, applying $\eta$ yields that $p \in \Sa(H)$. Hence, $\eta^{-1}(\set{\delta_H}) \cap X$ is a minimal proximal $H$-flow. As $H$ is strongly amenable it must be a singleton and thus $\eta|_X\colon X \to \Sa(G/H)$ is an isomorphism. It follows that $\eta\colon P(X) \to P(\Sa(G/H))$ is an isomorphism. 
\end{proof}
We now work towards the converse to the above theorem.
\begin{lemma} \label{UMPF-as-affine-subflow-lemma}
    We can identify $P(\Pi(G)) \subseteq P(\Sa(G))$ as an affine subflow. 
\end{lemma}
\begin{proof}
    There exists an extension $\psi\colon \Sa(G)\to \Pi(G)$ which in turn induces an affine extension $\phi\colon P(\Sa(G)) \to P(\Pi(G))$. Let $K \subseteq P(\Sa(G))$ be an affine subflow minimal with respect to $\phi(K)= P(\Pi(G))$. It follows that $\phi|_Z\colon  K \to P(\Pi(G))$ is an affine $G$-irreducible extension and thus an isomorphism. 
\end{proof}
\begin{question}
    The proof of Lemma \ref{UMPF-as-affine-subflow-lemma} can be easily adjusted to show $P(\Pi(G))$ is an affine subflow of $P(M(G))$. Thus there exists a probability measure $\mu $ on $M(G)$ such that $\overline{G\cdot \mu} \simeq \Pi(G)$. Can we explicitly construct or characterize such measures? 
\end{question}

If $X$ is a compact space, then the collection of all nonempty closed subsets of $X$, denoted $K(X)$, is itself a compact space with a basic open set given by
\[
\set{C \in K(X): C \cap U_j \neq \emptyset \text{ and } C \subseteq U \text{ for } j = 1,\dots, n}
\]
where $U,U_1,\dots, U_n \subseteq X$ are nonempty open subsets of $X$. Moreover, if $X$ is a $G$-flow, then $K(X)$ is also a $G$-flow with the action given by
\[
g\cdot A := gA = \set{ga:a \in A}
\]
for all $g\in G$ and $A \in K(X)$. By \cite[Proposition 5.10]{EllisEllis2014} we can explicitly see the action $\Sa(G) \curvearrowright K(X)$ via the \emph{circle operation} given by
\[
p \bigcdot A := \set{x\in X:\exists\text{ nets }(g_\lambda)_\lambda \subseteq G, (a_\lambda)_\lambda \subseteq A \text{ with } g_\lambda \to p \text{ and } g_\lambda a_\lambda \to x}.
\]

\begin{proof}[Proof of Theorem \ref{Thm:StrAmenSubgp} $(\Leftarrow)$]
    Consider $P(\Pi(H)) \subseteq P(\Sa(H)) \subseteq P(\Sa(G))$ and define the affine subflow
    \[
    Z:= \overline{\conv}(\overline{G\cdot \Pi(H)}) \subseteq P(\Sa(G)). 
    \]
    Let $\psi\colon P(\Sa(G)) \to P(\Sa(G/H))$ denote the affine extension given by $\psi(\delta_{g}) = \delta_{gH}$. Notice, as seen in the proof of Theorem \ref{Thm:AffProjCoverPointTrans}, we have that 
    \[
    \psi^{-1}(\delta_H) = P(\Sa(H)).
    \]
    We claim that $\phi:= \psi|_Z\colon  Z \to P(\Sa(G/H))$ is proximally irreducible. Notice $\phi$ is surjective as it contains $\Pi(H)$ which gets mapped to $\delta_H$. Now let $C \subseteq Z$ be an affine subflow such that 
    \[
    C \cap \conv(\overline{\ext}(Z) \cap \phi^{-1}(\set{\delta_y})) \neq \emptyset, \ \ \ \forall y\in \Sa(G/H).
    \]
    In particular, there exists positive reals $\lambda_1,\dots, \lambda_n$ such that $\sum_j \lambda_j = 1$ and there exists $x_1,\dots, x_n \in \overline{\ext}(Z) \cap \phi^{-1}(\set{\delta_H})$ with
    \[
    \lambda_1x_1 + \cdots  + \lambda_n x_n \in C.
    \]
    We claim that for all $1\leq j \leq n$ we have $x_j \in \Pi(H)$. First notice by Milman's theorem we have that
    \[
    \overline{\ext}(Z) \subseteq \overline{G\cdot \Pi(H)}.
    \]
    Fix $1\leq j \leq n$ and let $(g_\lambda)_\lambda \subseteq G$, $(y_\lambda)_\lambda \subseteq \Pi(H)$ be nets such that 
    \[
    x_j= \lim_\lambda g_\lambda y_\lambda.
    \]
    Applying $\phi$ and using the fact that each $y_\lambda \in P(\Sa(H))= \psi^{-1}(\set{\delta_H})$ we obtain that
    \[
    \delta_H = \lim_\lambda g_\lambda \delta_H.
    \]
    Let $(g_{\lambda_i})_i$ be a subnet such that $g_{\lambda_i} \to p \in \Sa(G)$. It follows that $\delta_H = p\delta_H$ and thus $p \in \Sa(H)$. Now by definition of the circle operation we obtain that
    \[
    x_j = \lim_i g_{\lambda_i}y_{\lambda_i} \in p\bigcdot \Pi(H).
    \]
    Let $(h_\gamma)_\gamma \subseteq H$ be a net such that $h_\gamma \to p$. We obtain that
    \[
    p\bigcdot \Pi(H)= \lim_\gamma h_\gamma\cdot \Pi(H) = \lim_\gamma \Pi(H) = \Pi(H).
    \]
    Therefore, $x_j \in \Pi(H)$ for all $1\leq j \leq n$. By proximality, there exists $q \in \Sa(H)$ such that $qx_1 = \cdots = qx_n \in \Pi(H)$. It follows that 
    \[
    qx_1 \in C
    \]
    and thus $C = Z$. Since $P(\Sa(G/H))$ was assumed to have no nontrivial proximally irreducible extensions we obtain that $\phi$ is an isomorphism and thus $\phi^{-1}(\set{\delta_H})$ consists of an $H$-fixed point, say $x\in \ext(Z)$. Now repeating the exact same argument as above we obtain that $x \in \Pi(H)$. Therefore, $H$ is strongly amenable. 
\end{proof}
Using the above theorem we can now characterize when the universal minimal proximal flow has a comeager orbit or is metrizable.

Recall that a closed subgroup $H \leq G$ is said to be \emph{presyndetic} if for all $U \in \mathcal{N}_G$ we have that $UH \subseteq G$ is a syndetic subset. This is seen to be equivalent to the minimality of the $G$-flow $\Sa(G/H)$. Moreover, $H$ is said to be \emph{co-precompact} if $\Sa(G/H) = \widehat{G/H}$. For Polish groups $G$, a closed subgroup $H$ is seen to be co-precompact if and only if $\Sa(G/H)$ is metrizable. For more information on presyndetic and co-precompact subgroups see \cite{BassoZucker2025}.
\begin{theorem} \label{Thm:StructureOfUMPF}
    Let $G$ be a Polish group.
    \begin{enumerate}
        \item $\Pi(G)$ has a comeager orbit if and only if there exists a presyndetic maximal strongly amenable closed subgroup $H \leq G$. Moreover in this case, $\Pi(G) \simeq \Sa(G/H)$.
        \item  $\Pi(G)$ is metrizable if and only if there exists a presyndetic co-precompact maximal strongly amenable closed subgroup $H \leq G$. Moreover in this case, $\Pi(G) \simeq \widehat{G/H}$. 
    \end{enumerate}
\end{theorem}
\begin{proof}
(1)
    $(\Rightarrow)$ As $\Pi(G)$ has a comeager orbit and is G-ED, by \cite[Theorem 5.5]{Zucker2021}, there exists a closed subgroup $H \leq G$ such that $\Pi(G) \simeq \Sa(G/H)$. By Proposition \ref{Prop:UMPFIsProxIrr}, $P(\Pi(G)) \simeq P(\Sa(G/H))$ is proximally irreducible and thus $H$ is strongly amenable by Theorem \ref{Thm:StrAmenSubgp}. Suppose $H' \leq G$ is a strongly amenable closed subgroup containing $H$. Consider $\phi\colon \Sa(G/H)\to \Sa(G/H')$ defined via $\phi(gH) = gH'$. It follows that $\phi$ is a proximal extension between minimal $G$-flows. By Proposition \ref{Prop:ProxExt}, the induced affine extension $\psi\colon  P(\Sa(G/H)) \to P(\Sa(G/H'))$ is proximally irreducible. Using Theorem \ref{Thm:StrAmenSubgp} along with the assumption that $H'$ is strongly amenable we obtain that $\psi$ must be an isomorphism. Therefore, we have that $H = H'$ and hence $H$ is a maximal strongly amenable closed subgroup. 
    
    $(\Leftarrow)$ Since $\Sa(G/H)$ is minimal we obtain, by \cite[Lemma 7.4]{Zucker2021}, that every proximal $G$-flow is also a proximal $H$-flow. Since $H$ is strongly amenable, every proximal $G$-flow has an $H$-fixed point. In particular, there exists an extension $\phi\colon \Sa(G/H) \to \Pi(G)$. By \cite[Proposition 14.1]{AngelKechrisLyons2014}, $\Pi(G)$ has a comeager orbit and $\Pi(G) \simeq \Sa(G/H')$ for some closed subgroup $H' \leq G$. Theorem \ref{Thm:StrAmenSubgp} along with Proposition \ref{Prop:ProxExt} tells us that $H'$ is strongly amenable. Moreover, by \cite[Proposition 6.7]{Zucker2021}, there exists $g\in G$ such that $H \subseteq gH'g^{-1}$. By maximality of $H$ we must have $H = H'$ and thus $\Sa(G/H) \simeq \Pi(G)$. 

    (2) Follows from the fact that $\Pi(G)$ is $G$-ED and that any minimal metrizable $G$-ED flow must have a comeager orbit (see \cite[Corollary 3.3]{Zucker2018}).
\end{proof}
\subsection{$G$-projectivity once again}
We now prove the claimed result about $G$-projective and affine $G$-projective covers mentioned after Theorems \ref{Thm:ProjCoverPointTrans} and \ref{Thm:AffProjCoverPointTrans}. 

A key technique behind the proofs will be the following lemma already implicitly used in the proof of Theorem \ref{Thm:StrAmenSubgp}.
\begin{lemma}\label{Lem:FiberStab}
    Let $G$ be a Polish group and $X$ a $G$-flow with a point $x_0 \in X$ having comeager orbit. If $Y$ is another $G$-flow, $\phi\colon Y \to X$ an extension, and $C \subseteq \phi^{-1}(\set{x_0})$ is a closed $\Stab(x_0)$-invariant subset, then
    \[
    \overline{G\cdot C} \cap \phi^{-1}(\set{x_0}) = C.
    \]
\end{lemma}
\begin{proof}
    Let $(g_\lambda)_\lambda \subseteq G$ and $(y_\lambda)_\lambda \subseteq C$ be nets such that 
    \[
   y:=  \lim_\lambda g_\lambda y_\lambda \in \phi^{-1}(\set{x_0}).
    \]
    Passing to a subnet we may assume $(g_\lambda)_\lambda$ converges to $p \in \Sa(G)$. Applying $\phi$ we obtain that
    \[
   x_0 =  \lim_\lambda g_\lambda x_0 = px_0
    \]
    and thus $p \in \Sa(\Stab(x_0))$ by Lemma \ref{Lem:Stab}. Moving the Vietoris Hyperspace and using that $C$ is $\Stab(x_0)$ invariant we obtain that
    \[
    y \in p \bigcdot C  \subseteq C.
    \]
\end{proof}
\begin{prop}\label{Prop:ProjCoverPointTransGen}
    Suppose $G$ is a Polish group, $X$ a point transitive $G$-flow, and $x_0 \in X$ has comeager orbit. Viewing $M(\Stab(x_0)) \subseteq \Sa(G)$ as a closed subset we obtain that 
    \[
    \proj_G(X) \simeq \overline{G\cdot M(\Stab(x_0))}.
    \]
\end{prop}
\begin{proof}
    Let $\phi\colon \Sa(G) \to X$ be the extension given by $\phi(g) = gx_0$ and $Y := \overline{G\cdot M(\Stab(x_0))}$. By the discussion in Section \ref{Sec:PointTrans} it suffices to prove that
    \[
    \psi:= \phi|_Y \colon Y \to X
    \]
    is a $G$-irreducible extension. Notice $\psi$ is surjective as 
    \[
    M(\Stab(x_0)) \subseteq \Sa(\Stab(x_0))=  \phi^{-1}(\set{x_0}).
    \]
    Let $C \subseteq Y$ be a subflow such that $\psi(C)= X$. It follows that $C \cap \psi^{-1}(\set{x_0}) \neq \emptyset$. By Lemma \ref{Lem:FiberStab}, we see that $\psi^{-1}(\set{x_0}) = M(\Stab(x_0))$. It follows that we must have $C= Y$.
\end{proof}
We now have the analogous result for affine $G$-projective covers.
\begin{prop}\label{Prop:AffProjCoverPointTransGen}
    Let $G$ be a Polish group and suppose $X$ is a $G$-flow such that $x_0 \in X$ has comeager orbit. Viewing $\Pi_s(\Stab(x_0)) \subseteq P(\Sa(G))$ as a closed subset we have that $\overline{\conv}(\overline{G\cdot\Pi_s(\Stab(x_0))})$ is isomorphic to the affine $G$-projective cover of $P(X)$.
\end{prop}
\begin{proof}
    Let $\psi\colon P(\Sa(G)) \to P(X)$ be the affine extension given by $\psi(\delta_{1_G}) = \delta_{x_0}$. By Theorem \ref{Thm:AltConOfAffProjCover} it suffices to prove that $\psi$ restricted to $$K:= \overline{\conv}(\overline{G\cdot\Pi_s(\Stab(x_0))}) \subseteq P(\Sa(G))$$ is affine $G$-irreducible. First notice that $\psi(\Pi_s(\Stab(x_0))) = \set{\delta_{x_0}}$ and thus $K$ maps onto $P(X)$. Next we claim that 
    \[
    \psi^{-1}(\set{\delta_{x_0}}) \cap K = \overline{\conv}(\Pi_s(\Stab(x_0))).
    \]
    As $\psi^{-1}(\set{\delta_{x_0}})\cap K$ is a face in $K$, let $y \in \psi^{-1}(\set{\delta_{x_0}})\cap K$ be an extreme point. By Milman's theorem we have that $y \in \overline{G\cdot \Pi_s(\Stab(x_0))}$. Then Lemma \ref{Lem:FiberStab} tells us that $y \in \Pi_s(\Stab(x_0))$. Therefore, we have that
    \[
    \psi^{-1}(\set{\delta_{x_0}}) \cap K \subseteq \overline{\conv}(\Pi_s(\Stab(x_0)))
    \]
    and the reverse inclusion is clear. 

    Let $C \subseteq K$ be an affine subflow such that $\psi(C)= P(X)$. Thus there exists 
    \[
    y\in C \cap \psi^{-1}(\set{\delta_{x_0}}) = C \cap \overline{\conv}(\Pi_s(\Stab(x_0))).
    \]
    Notice $\overline{\conv}(\Pi_s(\Stab(x_0)))$ is a proximal $\Stab(x_0)$-flow. Indeed, it is seen to be the image of $P(\Pi_s(\Stab(x_0)))$ under the barycenter map and proximality of $P(\Pi_s(\Stab(x_0)))$ follows as $\Pi_S(\Stab(x_0))$ is $\Stab(x_0)$-strongly proximal. Thus using proximality and minimality we obtain that $\Pi_s(\Stab(x_0)) \subseteq C$ and hence $C = K$. Therefore, $\psi|_K$ is an affine $G$-irreducible and our desired result follows.
\end{proof}
We also have a kind of converse to the above result.
\begin{prop}
    Let $G$ be a Polish group and $X$ a $G$-flow such that $x_0 \in X$ has comeager orbit. If $\phi\colon P(Y) \to P(X)$ is an affine $G$-irreducible extension, then $\phi^{-1}(\set{\delta_{x_0}}) \cap Y$ is a minimal strongly proximal $\Stab(x_0)$-flow.
\end{prop}
\begin{proof}
    By Lemma \ref{Lem:AffIrrImpliesIrr} it follows that $\phi|_Y\colon Y \to X$ is a $G$-irreducible extension. Thus each $y\in \phi^{-1}(\set{x_0}) \cap Y$ has dense $G$-orbit in $Y$. In particular, for any $y' \in \phi^{-1}(\set{x_0}) \cap Y$ there exists $p \in \Sa(G)$ with $py = y'$. Applying $\phi$ and using Lemma \ref{Lem:Stab} we obtain that $p \in \Sa(\Stab(x_0))$. Therefore, $\phi^{-1}(\set{x_0}) \cap Y$ is a minimal $\Stab(x_0)$-flow. Now given any $\mu \in P(\phi^{-1}(\set{x_0}) \cap Y) \subseteq P(Y)$ we see that 
    \[
    \phi(\overline{\conv}(\overline{G\cdot \mu})) = P(X)
    \]
    and thus $Y \subseteq \overline{G\cdot \mu}$ by Milman's theorem. Thus there exists $p \in \Sa(G)$ with $p \mu \in \phi^{-1}(\set{x_0}) \cap Y$. Applying $\phi$ once again we obtain that $p \in \Sa(\Stab(x_0))$, i.e., $\phi^{-1}(\set{x_0}) \cap Y$ is also a strongly proximal $\Stab(x_0)$-flow.
\end{proof}
\subsection{Maximal proximally irreducible extensions}
 A natural question is if every affine $G$-flow has a proximally $G$-projective cover. In particular, given an affine $G$-flow $L$ does there exists a maximal proximally irreducible extension $K$ of $L$?

An obstacle to the above is that it is not clear if the composition of two proximally irreducible extensions remains proximally irreducible. We state this as an explicit question.
\begin{question}\label{Ques:CompProxIrrExt}
    Let $K,L$, and $P$ be affine $G$-flows. If $\phi\colon L \to K$ and $\psi\colon P \to L$ are proximally irreducible extensions, then is $\phi \circ \psi \colon P \to K$ also proximal irreducible?
\end{question}

Regardless we can prove such maximal proximally irreducible extensions exist in certain cases when said obstacle disappears. 
\begin{prop}\label{Prop:CompOfProxIrr}
    Let $X$ be a $G$-flow such that either it is minimal or it contains a comeager orbit. If $L,K$ are affine $G$-flows, $\phi\colon L \to P(X)$ and $\psi\colon  K \to L$ proximally irreducible extensions, then $\phi \circ \psi\colon K \to P(X)$ is proximally irreducible. 
\end{prop}
\begin{proof}
    First assume $X$ is a minimal $G$-flow. Let $C \subseteq K$ be an affine subflow such that 
    \[
    C \cap \conv(\overline{\ext}(K)\cap \psi^{-1}(\phi^{-1}(\set{x}))) \neq \emptyset, \ \ \ \ \text{for all }x\in X.
    \]
    Given $z\in \overline{\ext}(L)$ notice that any $x,y\in \phi^{-1}(\set{\phi(z)})$ are a proximal pair. Indeed, consider the affine subflow
    \[
     \overline{\conv}(\overline{G\cdot (\frac{x + y}{2})}) \subseteq L.
    \]
    Using minimality of $X$, the assumption that $\phi$ is proximally irreducible, and Milman's theorem we obtain that $$\ext(L) \subseteq \overline{G\cdot (\frac{x + y}{2})}.$$ In particular, there exists $p\in \Sa(G)$ such that $px = py \in \ext(L)$. 
    Now by assumption, there exists $y_1,\dots,y_n\in \overline{\ext}(K)$ and positive $\lambda_1,\dots, \lambda_n \in \R$ such that $\psi(y_j) \in \phi^{-1}(\set{\phi(z)})$ for all $1\leq j \leq n$, $\sum_{j=1}^n \lambda_j = 1$, and 
    \[
    w:= \lambda_1y_1 + \cdots + \lambda_n y_n \in C.
    \]
    Pick $p \in \Sa(G)$ such that $p\psi(y_1) = \cdots  = p\psi(y_n) = z$. It follows that
    \[
    pw \in C\cap \conv (\overline{\ext}(K) \cap \psi^{-1}(\set{z})).
    \]
    Using that $\psi$ is proximally irreducible we obtain that $C = K$ thus showing $\phi\circ \psi$ is proximally irreducible. 

    Now assume $X$ is a $G$-flow such that $x_0 \in X$ has comeager orbit. We first claim that $\phi^{-1}(\set{x_0}) \cap \overline{\ext}(L)$ is a minimal proximal $\operatorname{Stab}(x_0)$-flow. Indeed, if $z \in \phi^{-1}(\set{x_0})$, then using that $x_0\in X$ has dense orbit we get that the affine subflow $C := \overline{\conv}(\overline{G\cdot z})$ satisfies
    \[
    C \cap \conv (\overline{\ext}(L)\cap\phi^{-1}(x)) \neq \emptyset, \ \ \ \text{for all }x\in X.
    \]
    Thus $C = L$ and Milman's theorem implies $\overline{G\cdot z}  = \overline{\ext}(L)$. In particular, given any $z' \in \phi^{-1}(\set{x_0}) \cap \overline{\ext}(L)$ there exists $p \in \Sa(G)$ with $pz = z'$. Applying $\phi$ and using Lemma \ref{Lem:Stab} we obtain that $p \in \Sa(\Stab(x_0))$. It follows that $\phi^{-1}(\set{x_0})\cap \overline{\ext}(Z)$ is a minimal $\operatorname{Stab}(x_0)$-flow. Given $z_1,z_2 \in \phi^{-1}(\set{x_0})\cap \overline{\ext}(L)$, the exact same argument with 
    \[
    C := \overline{\conv}(\overline{G\cdot (\frac{z_1 + z_2}{2})})
    \]
    implies that $ \phi^{-1}(\set{x_0})\cap \overline{\ext}(L)$ is a proximal $\Stab(x_0)$-flow.
    
    Suppose $C \subseteq Y$ is an affine subflow such that
    \[
    C \cap \conv(\overline{\ext}(K)\cap \psi^{-1}(\phi^{-1}(\set{x}))) \neq \emptyset,  \ \ \ \text{for all }x\in X.
    \]
    Let $w:= \lambda_1y_1 + \cdots + \lambda_ny_n\in C$ be such that $y_j\in \overline{\ext}(Y)$, $\psi(y_j) \in \phi^{-1}(\set{x_0})$ for $1\leq j \leq n$ and $\lambda_1,\dots, \lambda_n$ are positive reals with $\sum_{j=1}^n \lambda_j = 1$. Now we can find $p\in \Sa(G)$ such that
    \[
    z_0  := p\psi(y_1) = \cdots p \psi(y_n) \in \phi^{-1}(\set{x_0}).
    \]
    Now using that $z_0 \in \overline{\ext}(L)$ has dense orbit we obtain that
    \[
    C \cap \conv (\overline{\ext}(K) \cap \psi^{-1}(\set{z})) \neq \emptyset, \ \ \ \text{for all }z \in \overline{\ext}(L).
    \]
    Therefore, $C = K$ showing $\phi \circ \psi$ is proximally irreducible. 
\end{proof}

   \begin{lemma}\label{Lem:FactorOfProxIrr}
    Let $K$, $L$ and $P$ be affine $G$-flows. Suppose we have affine extensions $\phi\colon K \to L$, $\psi\colon  P \to L$, and $\eta\colon K \to P$ such that $\phi = \psi \circ \eta$. If $\phi$ is a proximally irreducible extension, then so is $\eta$. 
\end{lemma}
\begin{proof}    
    Suppose $C \subseteq K$ is an affine subflow satisfying 
    \[
     C \cap \conv (\overline{\ext}(K)\cap \eta^{-1}(\set{z})) \neq \emptyset, \ \ \ \ \text{for all } z\in \overline{\ext}(P). 
    \]
    Given $y\in \overline{\ext}(L)$ pick $z \in \overline{\ext}(P)$ such that $\psi(z) = y$. Using $\phi = \psi \circ \eta$ we obtain that
    \[
    \conv (\overline{\ext}(K)\cap \eta^{-1}(\set{z})) \subseteq \conv (\overline{\ext}(K)\cap \phi^{-1}(\set{y})).
    \]
    As $\phi$ is proximally irreducible we obtain that $C = K$, thus showing $\eta$ is proximally irreducible. 
\end{proof}
\begin{lemma}\label{Lem:InvLimProxIrr}
   Let $L$ be an affine $G$-flow and $(K_\lambda)_{\lambda \in I}$ be a collection of affine $G$-flows where $I$ is some directed set. Suppose we have proximally irreducible extensions $\pi_\lambda\colon  K_\lambda \to L$ and we have affine extension $\iota_\lambda^\beta\colon  K_\beta \to K_\lambda$ such that $\pi_\beta = \pi_\lambda \circ \iota_\lambda^\beta$ for all $\lambda < \beta \in I$. If $K:= \varprojlim K_\lambda$ is the inverse limit and $\eta_\lambda\colon  K \to K_\lambda$ are the canonical affine extensions, then defining $\pi:= \pi_\lambda \circ \eta_\lambda$ (choice of $\lambda$ does not matter) we have that $\pi\colon  K \to L$ is a proximally irreducible extension. 
\end{lemma}
\begin{proof}
       Given any $\lambda < \beta \in I$ we have that $\pi_\beta = \pi_\lambda \circ \iota_\lambda^\beta$ where $\pi_\beta$ is proximally irreducible. By Lemma \ref{Lem:FactorOfProxIrr}, we have that $\iota_\lambda^\beta$ is also proximally irreducible. In particular, $\iota_\lambda^\beta(\overline{\ext}(K_\beta)) = \overline{\ext}(K_\lambda)$. 
       
        We claim that for all $\lambda \in I$ we have $\eta_\lambda(\overline{\ext}(K))= \overline{\ext}(K_\lambda)$. Defining the $G$-flow $C_\lambda := \eta_\lambda^{-1}(\overline{\ext}(K_\lambda)) \cap \overline{\ext}(K)$ we have that $\eta_\lambda(C_\lambda) = \overline{\ext}(K_\lambda)$. Given $\lambda < \beta \in I$ and using that $\eta_\lambda = \iota_\lambda^\beta \circ \eta_\beta$ we obtain
        \begin{align*}
            C_\beta &= \eta_\beta^{-1}(\overline{\ext}(K_\beta)) \cap \overline{\ext}(K) \\
            &= \eta_\beta^{-1}((\iota_\lambda^\beta)^{-1}(\overline{\ext}(K_\lambda))\cap \overline{\ext}(K_\beta) ) \cap \overline{\ext}(K) \\
            &\subseteq \eta_\beta^{-1}((\iota_\lambda^\beta)^{-1}(\overline{\ext}(K_\lambda))) \cap \overline{\ext}(K) \\
            &= \eta_\lambda^{-1}(\overline{\ext}(K_\lambda)) \cap \overline{\ext}(K) \\
            &= C_\lambda.
        \end{align*}
        Moreover, we have 
        \[
        \eta_\lambda(C_\beta) = \iota_\lambda^\beta(\eta_\beta(C_\beta)) = \iota_\lambda^\beta(\overline{\ext}(K_\beta))= \overline{\ext}(K_\lambda). 
        \]
        Therefore the $G$-flow $C:= \bigcap_{\lambda\in I} C_\lambda \subseteq \overline{\ext}(K)$ satisfies $\eta_\lambda(C) = \overline{\ext}(K_\lambda)$ for all $\lambda \in K$. In particular, we obtain 
        \[
        \eta_\lambda(\overline{\conv}(C)) = K_\lambda
        \]
        for all $\lambda \in I$. Hence, by the universal property of the inverse limit we must have $\overline{\conv}(C) = K$ and thus by Milman's theorem $\overline{\ext}(K) = C$. From this it follows that $\eta_\lambda(\overline{\ext}(K))= \overline{\ext}(K_\lambda)$ for all $\lambda \in I$. 
        
        We now show $\pi\colon  K \to L$ is a proximally irreducible extension. Let $Z \subseteq K$ be an affine subflow such that 
        \[
         Z \cap \conv (\overline{\ext}(K)\cap \pi^{-1}(\set{y})) \neq \emptyset, \  \ \ \text{for all } y\in \overline{\ext}(L).
        \]
       Fix $\lambda \in I$ and $z\in \overline{\ext}(K_\lambda)$. Since $\pi_\lambda\colon K_\lambda \to L$ is proximally irreducible we have that $\pi_\lambda(\overline{\ext}(K_\lambda)) = \overline{\ext}(L)$. Let $w := \alpha_1x_1  + \cdots + \alpha_n x_n \in Z$ where $x_1,\dots, x_n \in \overline{\ext}(K)$ are such that $\pi(x_j) = \pi_\lambda(z) \in \overline{\ext}(L)$ for $1\leq j\leq n$ and $\alpha_1,\dots, \alpha_n$ are positive real numbers with $\sum_{j=1}^n \alpha_j = 1$. As $\pi_\lambda\circ\eta_\lambda = \pi$ we obtain that for all $1\leq j\leq n$, $\eta_\lambda(x_j) \in \pi_\lambda^{-1}(\set{y})$ and thus 
       \[
       \eta_\lambda(w) \in \eta_\lambda(Z) \cap \conv (\overline{\ext}(K_\lambda) \cap \pi_\lambda^{-1}(\set{y})) \neq \emptyset.
       \]
       Since this holds for all $y\in \overline{\ext}(L)$ we must have $\eta_\lambda(Z) = K_\lambda$. Using the universal property of the inverse limit we obtain $Z = K$. Therefore, $\pi\colon K \to L$ is proximally irreducible.
\end{proof}
\begin{theorem}\label{Thm:MaxProxIrrFactor}
        Let $K,L$ be affine $G$-flows and $\phi\colon K\to L$ an affine extension. Then there exists an affine $G$-flow $P$ such that the following holds.
        \begin{enumerate}
            \item There exists an affine extension $\psi\colon  K \to P$.
            \item There exists a proximally irreducible extension $\pi\colon  P \to L$ such that $\pi \circ \psi = \phi$.
            \item $P$ is a maximal factor of $K$ with respect to the above two properties.
        \end{enumerate}
    \end{theorem}
    \begin{proof}
        Notice $L$ itself satisfies conditions $1$, $2$. Now suppose $(R_\lambda)_{\lambda \in I}$ are closed decreasing $G$-invariant equivalence relations on $K$ such that $K/R_{\lambda}$ are affine $G$-flows and there exists proximally irreducible extensions $\pi_\lambda\colon K/R_\lambda \to L$. Furthermore all the affine extensions $K \to K/R_\lambda$, $K/R_\beta \to K/R_\lambda$ for $\lambda < \beta$, and $K/R_\lambda \to L$ are compatible in the way one expects. Then let $R:= \bigcap_{\lambda \in I} R_\lambda$ and notice 
        \[
        K/R \simeq \varprojlim K/R_\lambda.
        \]
        By Lemma \ref{Lem:InvLimProxIrr}, $K/R$ satisfies the above two conditions. By Zorn's lemma such a maximal factor of $K$ exists.
    \end{proof}
We can now show certain affine $G$-flows have a maximal proximally irreducible extension. In particular, the following theorem holds for affine $G$-flows of the form $P(Y)$ where $Y$ is either a minimal $G$-flow or contains a comeager orbit.
 \begin{theorem}\label{Thm:MaxProxIrrExt}
        Let $L$ be an affine $G$-flow such that if $K,P$ are affine $G$-flows and $\phi\colon K \to L$, $\psi\colon  P \to K$ are proximally irreducible extensions, then $\phi \circ \psi \colon P \to L$ is also a proximally irreducible extension. Then there exists a proximally irreducible affine  $G$-flow $P(X)$ and a proximally irreducible extension $\phi\colon P(X) \to L$, i.e., $L$ admits a maximal proximally irreducible extension.
    \end{theorem}
    \begin{proof}
        Let $Z := \proj_G(\overline{\ext}(L))$ be the $G$-projective cover and $\psi\colon  P(Z) \to Y$ the induced affine extension. Let $K$ be an affine factor of $P(Z)$ and $\phi\colon K \to L$ a proximally irreducible extension as given by Theorem \ref{Thm:MaxProxIrrFactor}. Suppose $C$ is an affine $G$-flow and $\eta\colon C \to K$ is a proximally irreducible extension. Now $\phi\circ \eta\colon C \to L$ is proximally irreducible and thus incudes a $G$-projective extension $$\phi\circ \eta|_{\overline{\ext}(C)}\colon \overline{\ext}(C) \to\overline{\ext}(L).$$Since $Z$ is the $G$-projective cover of $\overline{\ext}(L)$, there exists an extension $f\colon  Z \to \overline{\ext}(C)$ such that $\phi\circ \eta \circ f = \psi|_Z$. Consider the induced affine extension $\Tilde{f}\colon P(Z) \to P(\overline{\ext}(C))$. Now compose with the barycenter map to yield an affine extension $h\colon P(Z)\to C$ such that $\phi\circ \eta \circ h = \psi$. By our assumption on $L$, $\phi \circ \eta$ is a proximally irreducible extension. Now Theorem \ref{Thm:MaxProxIrrFactor} tells us that $\eta\colon C \to K$ must be an isomorphism. Therefore, $K$ is proximally irreducible and thus $K \simeq P(X)$ for some $G$-flow $X$.
    \end{proof}
  We can also show that if such a maximal proximally irreducible extension does exist, then uniqueness follows automatically.
      \begin{prop}\label{Prop:UniMaxProxIrrExt}
        Let $L$ be an affine $G$-flow. If $P(X)$ is a proximally irreducible flow and a proximally irreducible extension of $L$, then $X$ is unique up to isomorphism. 
    \end{prop}
    \begin{proof}
       Let $\phi\colon P(X) \to L$ and $\psi\colon P(Z) \to L$ be maximal proximally irreducible extensions.  Define $C \subseteq P(X) \times P(Z)$ via
       \[
       C:= \set{(x,z): \phi(x) = \psi(z)}.
       \]
       It follows that $C$ is an affine subflow and the projection map $\pi_1\colon C \to P(X)$ is an affine extension. Let $C' \subseteq C$ be an affine subflow minimal with respect to $\pi_1(C') = P(X)$. Thus $\pi_1|_{C'}\colon  C' \to P(X)$ is an affine $G$-irreducible extension. In particular, $\pi_1|_{C'}$ is a proximally irreducible extension and hence a homeomorphism. If $\pi_2\colon C \to P(Z)$ denotes the projection onto $P(Z)$, then $\eta:= \pi_2\circ (\pi_1|_{C'})^{-1}\colon  P(X) \to P(Z)$ is an affine $G$-map such that $\psi \circ \eta = \phi$. It follows by Lemma \ref{Lem:FactorOfProxIrr}, that $\eta$ is a proximally irreducible extension. Therefore, $\eta\colon P(X) \to P(Z)$ is an isomorphism since $P(Z)$ was assumed to have no nontrivial proximally irreducible extensions. 
    \end{proof}
    Now using Theorem \ref{Thm:MaxProxIrrExt} we can prove an analogue of Proposition \ref{Prop:MaxStrProxExt} for the case of proximal extensions. Note the proximal case is already know (see \cite[Chapter 10]{Auslander1988}) but our techniques developed here provide another proof. 
    \begin{prop}\label{Prop:MaxProxExt}
        Let $Y$ be a minimal $G$-flow. Then there exists a minimal $G$-flow $X$ and a proximal extension $\phi\colon X \to Y$ such that for any minimal $G$-flow $Z$ and a proximal extension $\psi\colon Z \to Y$, there exists a proximal extension $\pi\colon X \to Z$ with $\psi \circ \pi = \phi$, i.e., $X$ is seen a maximal proximal extension of $Y$. Moreover, $X$ is unique up to isomorphism.
    \end{prop}
    \begin{proof}
        By Proposition \ref{Prop:CompOfProxIrr} we see that $P(Y)$ satisfies the assumptions of Theorem \ref{Thm:MaxProxIrrExt}. Thus there exists a maximal proximally irreducible extension $\eta\colon P(X) \to P(Y)$. Letting $\phi:= \eta|_X$ we see that Proposition \ref{Prop:ProxExt} implies $\phi\colon X\to  Y$ is a proximal extension. Given any minimal $G$-flow $Z$ and a proximal extension $\psi\colon Z \to Y$ we obtain that the induced affine extension $\Tilde{\psi}\colon P(Z) \to P(Y)$ is proximally irreducible. By the proof of Proposition \ref{Prop:UniMaxProxIrrExt} there exists a proximally irreducible extension $\theta: P(X) \to P(Z)$ such that $\Tilde{\psi}\circ \theta = \eta$. It follows that $\pi := \theta|_X$ is the required proximal extension. Moreover, uniqueness of $X$ also follows from Proposition \ref{Prop:UniMaxProxIrrExt}.
    \end{proof}
    We now prove analogies of Theorem \ref{Thm:AffProjCoverPointTrans} and Proposition \ref{Prop:AffProjCoverPointTransGen} for proximally irreducible extensions. 
    \begin{theorem}\label{Thm:MaxProxIrrOfPointTrans}
        Let $G$ be a Polish group and $X$ a $G$-flow such that $x_0 \in X$ has comeager orbit. If the universal minimal proximal flow of $H :=\Stab(x_0)$ contains a comeager orbit, then $P(\Sa(G/H^*))$ is isomorphic to the maximal proximally irreducible extension of $P(X)$ where $H^* \leq H$ is a maximal strongly amenable closed subgroup such that $\Pi(H) \simeq \Sa(H/H^*)$.
    \end{theorem}
    \begin{proof}
       The existence of such an $H^* \leq H$ follows by Theorem \ref{Thm:StructureOfUMPF}. By Theorem \ref{Thm:StrAmenSubgp} we see that $P(\Sa(G/H^*))$ is proximally irreducible. Thus it suffices to show that the canonical affine extension $\phi\colon P(\Sa(G/H^*)) \to P(X)$ given by $\phi(gH^*) = g\delta_{x_0}$ is proximally irreducible. As in the proof of Theorem \ref{Thm:AffProjCoverPointTrans} we see that
        \[
        \phi^{-1}(\set{\delta_{x_0}}) = P(\Sa(H/H^*)).
        \]
        Now suppose $C \subseteq P(\Sa(G/H^*))$ is an affine subflow such that
        \[
        C \cap \conv(\Sa(G/H)\cap \phi^{-1}(\set{\delta_x})) \neq \emptyset, \ \ \ \text{for all }x \in X.
        \]
        In particular, there exists positive reals $\lambda_1,\dots, \lambda_n \in \R$ with $\sum_{j=1}^n \lambda_j = 1$ and there exists $y_1,\dots, y_n \in \Sa(G/H^*)\cap \phi^{-1}(\set{\delta_{x_0}}) = \Sa(H/H^*)$ such that 
        \[
        \lambda_1y_1 + \cdots + \lambda_n y_n \in C.
        \]
        Using $H$-proximality and $H$-minimality of $\Sa(H/H^*)$ we see that $\Sa(H/H^*) \subseteq C$ and thus $C = P(\Sa(G/H^*))$. Therefore, $\phi$ is proximally irreducible and the desired result follows. 
    \end{proof}
    \begin{prop}
           Let $G$ be a Polish group and $X$ a $G$-flow such that $x_0 \in X$ has comeager orbit. Viewing $\Pi(H) \subseteq P(\Sa(H)) \subseteq P(\Sa(G))$ as a closed subset where $H := \Stab(x_0)$ we have that $\overline{\conv}(\overline{G\cdot \Pi(H)})$ is isomorphic to the maximal proximally irreducible extension of $P(X)$.
    \end{prop}
    \begin{proof}
        Let $\phi\colon P(\Sa(G)) \to P(X)$ be the affine extension given by $\phi(\delta_{g}) = \delta_{gx_0}$. We claim that it suffices to prove $\phi$ restricted to $K := \overline{\conv}(\overline{G\cdot \Pi(H)})$ is proximally irreducible. Indeed, by Theorem \ref{Thm:MaxProxIrrExt}, let $\psi\colon P(Y)\to P(X)$ be the maximal proximally irreducible extension of $P(X)$. It is not hard to show there exists an affine extension $\pi\colon P(\Sa(G)) \to P(Y)$ such that $\phi = \psi \circ \pi$. As $\phi|_K$ is proximally irreducible, Lemma \ref{Lem:FactorOfProxIrr} tells us that $\pi|_K$ is also proximally irreducible and thus an isomorphism. 

       One can use the exact same argument as in proof of the converse direction of Theorem \ref{Thm:StrAmenSubgp} to show $\eta$ is proximally irreducible. 
    \end{proof}

\bibliographystyle{alpha}

\end{document}